\documentclass[11pt]{amsart}
\usepackage[foot]{amsaddr}

\usepackage{graphicx} 
\usepackage{amsmath,amssymb}
\usepackage{bm}
\usepackage{tikz}
\usepackage{xcolor}
\usepackage{pgfplots}
\usepackage{mathrsfs}
\usepackage{subfigure}
\usepackage{verbatim}
\usepackage{float}
\usepackage{setspace}
\usepackage{amsthm}
\usepackage{geometry}
\usepackage{color}
\usepackage{esint}
\usepackage{marginnote}
\theoremstyle{definition}
\usepackage{bm}
\usepackage{lineno}
\usepackage[sort&compress]{natbib} 
\usepackage{xcolor} 
\usepackage{hyperref}

\hypersetup{
    colorlinks=true,
    linkcolor=blue, 
    citecolor=green, 
    urlcolor=red 
}
\usepackage{cleveref}
\usepackage{enumerate}

\DeclareMathOperator{\st}{star}

\newtheorem{definition}{Definition}[section]
\newtheorem{proposition}{Proposition}[section]
\newtheorem{theorem}{Theorem}[section]
\newtheorem{lemma}{Lemma}[section]

\theoremstyle{remark}
\newtheorem{remark}{Remark}[section]
\newtheorem{corollary}{Corollary}[section]

\newcommand{\normmm}[1]{{\left\vert\kern-0.25ex\left\vert\kern-0.25ex\left\vert #1 
		\right\vert\kern-0.25ex\right\vert\kern-0.25ex\right\vert}}
  \DeclareMathOperator{\divergence}{div}
  \renewcommand{\div}{\divergence}
  \DeclareMathOperator{\curl}{curl}
  \makeatletter
  \@addtoreset{equation}{section}
  \makeatother

\usetikzlibrary{calc}
\usetikzlibrary{decorations.pathreplacing}
\begin{document}

\title[FEM for stress gradient elasticity]{Mixed Finite element method for stress gradient elasticity}

\author {Ting Lin}
\address{School of Mathematical Sciences, Peking University, Beijing 100871, P.R. China.}
\email{ lintingsms@pku.edu.cn }

\author {Shudan Tian}
\address{School of Mathematics and Computational Science, Xiangtan University, Xiangtan 411105, P. R. China. }
\email{ shudan.tian@xtu.edu.cn }

\thanks{The work of TL was supported by the NSFC Project 123B2014. The work of ST was supported by  the NSFC 12401483. }


\begin{abstract}
This paper develops stable finite element pairs for the linear stress gradient elasticity model, overcoming classical elasticity's limitations in capturing size effects. We analyze mesh conditions to establish parameter-robust error estimates for the proposed pairs, achieving unconditional stability for finite elements with higher vertex continuity and conditional stability for CG-DG pairs when no interior vertex has edges lying on three or fewer lines. Numerical experiments validate the theoretical results, demonstrating optimal convergence rates.
\end{abstract}



\maketitle

\def\a#1{\begin{align*}#1\end{align*}}\def\an#1{\begin{align}#1\end{align}}
\def\ad#1{\begin{aligned}#1\end{aligned}}

\section{Introduction}
Classical linear elasticity models have been widely applied to fundamental problems and engineering applications. As a consequence, numerous structure-preserving and efficient finite element methods for these models are proposed and well-developed in literature, such as  \cite{Arnold2002,Arnoldweaksym,Arnold3D,HuZhang14,HZ3D,HZfamily,HZlowest,GuzmanAfiled,PaulyZulehner23}. Nevertheless, for materials with microstructure, the classical elasticity model lacks an internal length scale in its constitutive structure, making it unable to capture phenomena such as size effects \cite{Aifantis2009,Aifantis2011overview}. Consequently, classical continuum theories fall short in providing an accurate and detailed description of deformation phenomena in such materials observed in experiments \cite{Aifantis2011overview}. This motivates further exploration of the generalization either by introducing modifying terms or by allowing certain laws to be violated.

Gradient elasticity theories \cite{MINDLIN1964,Aifantis1986} extend classical elasticity by incorporating higher-order spatial derivatives of strain, stress, etc. The introduction of higher-order terms of strain or stress can handle the singularity caused by the crack tips and dislocation cores \cite{Aifantis2011overview}. In this article, we focus on the linear stress gradient elasticity model, first introduced by Eringen \cite{Eringen1983}, as a model problem of the gradient elasticity theory. In the setting of stress gradient theory, the constitution law is replaced by $\mathcal E = \mathcal A(\sigma - \iota^2 \Delta \sigma)$, where the size parameter $\iota$ is a positive constant and we assue that $\iota \le 1$ in this paper. 
This problem has been discussed in engineering~\cite{Aifantis2011overview,VariationLSG}, but
 there are few relevant theoretical results in numerical analysis so far. 

Given a two-dimensional connected domain $\Omega$, the linear stress gradient elasticity equation is 
\begin{equation}
	\begin{cases}
\mathcal{E} = \mathcal{A} ( \sigma-\iota^2\triangle\sigma ) & \text{in }\Omega,  \\
\mathrm{div} \sigma =  f & \text{in }\Omega,  \\
2\mathcal{E}=\mathrm{symgrad}~ u & \text{in }\Omega,  \\
 \sigma  n =  g_f,~&\text{on }\Gamma_f,\\
 u =  g_c,~&\text{on } \Gamma_c ,\\
\partial_n\sigma = 0,&\text{on }\partial\Omega.
	\end{cases}
	\label{lsg-strong}
\end{equation}
where $\Gamma_c$ and $\Gamma_f $ is a partition of $\partial \Omega$.
Here $\sigma$, $\mathcal{E}$, and $u$ represent the stress (symmetric matrix-valued), strain (symmetric matrix-valued), and displacement (vector-valued), respectively. In this paper, we assume $\mathcal A$ is an isotropic fourth-order tensor, defined as 
$$
 \mathcal{A}\sigma = \frac{1}{2\mu}\sigma-\frac{\lambda}{2\mu(2\mu+2\lambda)}\mathrm{tr}(\sigma)I.
$$ Clearly, $\mathcal A$ models the relationship between the stress and the modified strain.



To analyze the linear stress gradient elasticity equation, we introduce the variational formulation of \eqref{lsg-strong} in a fixed form. Mathematically speaking, the variational formulation of the linear stress gradient model \cite{VariationLSG} can be regarded as a perturbed Hellinger-Reissner elasticity variation. Let $\mathbb S^2$ be the space of two-dimensional symmetric matrices, $\mathbb R^2$ be the space of two-dimensional vectors. 
Given $\Sigma \subset H^1(\Omega; \mathbb S^2)$ and vector-valued space $Q \subset L^2(\Omega; \mathbb R^2)$, the variational framework seeks solution $ \sigma \in \Sigma,~ u \in Q$ such that
\begin{equation}
	\begin{cases}
	&\iota^2(\nabla  \mathcal{A} \sigma,\nabla  \tau)+( \mathcal{A}\sigma, \tau) + (\mathrm{div} \tau, u) = 0,~\forall  \tau \in \Sigma,\\
	&(\mathrm{div} \sigma, v) = ( f, v),~\forall v\in Q.
	\end{cases}
\label{lsg-weak}
\end{equation}

For simplicity, here we only consider the planar case with the natural boundary conditions, i.e., $\Gamma_f = \emptyset.$ In this case, we set $\Sigma := H^1(\Omega; \mathbb S^2)$ and $Q := L^2(\Omega; \mathbb R^2)$. The stress and displacement can be put in to the last two slots of the smoothed stress complex:
\begin{equation}
\label{eq:complex-sobolev}
\mathbf{P}_1 ~\hookrightarrow~ H^3(\Omega) \xrightarrow{\curl \curl^T} H^1(\Omega;\mathbb{S}^2) \xrightarrow{\div} L^2(\Omega;\mathbb R^2) \to 0,
\end{equation}
where $\curl \curl^T$ denotes the planar Airy tensor. The two parameters $\lambda$ and $\iota$ are essential in modeling. 
 When $\lambda \to +\infty$, Poisson's ratio $\nu = \frac{\lambda}{2(\mu + \lambda)} \to \frac{1}{2}$, the materials become incompressible. When $\iota \to 0$, the effects introduced by the gradient are eliminated, and the model will converge to the standard linear elasticity model. Therefore, the goal of this paper is to provide  $\lambda$- and $\iota$-robust numerical schemes. 
 
It should be highlighted that the stress gradient elasticity is selected as one of the model problems in gradient elasticity theory. In fact, gradient elasticity theory encompasses various formulations. Notably, the strain gradient model adopts a complementary constitutive law, expressed as $\mathcal{E} - \iota^2 \Delta \mathcal{E} = \mathcal{A}\sigma$ \cite{LSG1992}. This formulation results in a variational framework that yields a singularly perturbed fourth-order equation in the primal form \cite{Aifantis1999,Ming2017,MingH2korn}. The numerical treatment can be found in \cite{MingMixLSG,HuangLSGSIAM,Huang2025LSG}. In contrast, the stress gradient model considered in this paper produces a perturbed equation in mixed form, which is not fully addressed in the literature. Therefore, the analysis and the numerical scheme of the linear stress gradient elasticity problem can also contributes to the understanding of singular perturbation mixed problems.

\textbf{Contributions.} This paper develops stable finite element pairs for the stress gradient elasticity model \eqref{lsg-weak}, ensuring parameter-robust error estimates with respect to $\lambda$ and $\iota$. Especially, the results are well-adapted when $\iota = 0$, leading to some stable pairs for the linear elasticity with $H^1$ conforming stress.


By the Ladyzhenskaya--Babu\v{s}ka--Brezzi conditions, it then suffices to consider $\Sigma_h \times Q_h \subset H^1(\Omega;\mathbb S^2) \times L^2(\Omega;\mathbb R^2)$ such that {(1) (closedness)} $\div \Sigma_h = Q_h$, and for any $q_h \in Q_h$, {(2) (existence of Fortin operator)} there exists $\sigma_h \in \Sigma_h$ such that $\div \sigma_h = q_h$ and $\|\sigma_h\|_{H^1} \lesssim \|q_h\|_{L^2}$. Throughout this paper, we call the pair is \emph{stable} if the above conditions hold.  Generally speaking, the goal can be separated into two parts:
\begin{enumerate}
	\item When $\div :\Sigma_h \to Q_h$ is surjective. That is, $\div \Sigma_h = Q_h$. 
	\item When can we find the bounded inverse for the pair $(\Sigma_h, \div \Sigma_h$). 
\end{enumerate}

Let us first introduce the candidate spaces for discretization. 
The most natural choice is to use continuous piecewise $P_k$ spaces $\Sigma_k^{0}$ (referred as Continuous Galerkin elements or Lagrange elements) as the discrete stress spaces, and to use discontinuous piecewise $P_{k-1}$ spaces $Q_{k-1}^{-1}$ (referred as Discontinuous Galerkin elements) as the discrete displacement spaces, to meet the regularity requirement. Moreover, the CG-DG pairs also yield great flexibility in implementation, leading to computational efficiency. The starting point of our investigation is then to design a stable CG-DG pair. 

This paper establishes the first results for the stable pair $H^1(\mathbb{S}^2) \times L^2(\mathbb{R}^2)$ based on the CG-DG pair, providing a foundation for designing stable pairs. 
Specifically, we prove that if \emph{for any interior vertex, all of its edges lie in not less than four lines,}  then $\div: \Sigma_k^0 \to Q_{k-1}^{-1}$ is surjective and a bounded right inverse exists, provided $k\ge 7$, see \Cref{thm:overall} for the precise statement. The violation can happen in two scenarios, one is the criss-cross, and the other is three-line-cross, see \Cref{fig:singularities} for an illustration. The presence of two types of singularities causes rank deficiencies in the divergence operator, affecting stability of the finite element pairs. {Notably, the type II singularities seems less discussed to the mixed finite element methods.}

\begin{figure}[htbp]
		\begin{tikzpicture}[scale=1.8]
			\draw[thick] (1/2,1/2) -- (1/2,1);
	   	    \draw[thick] (1/2,1/2) -- (1,1/2);
	   	    \draw[thick] (1/2,1/2) -- (1/2,0);
	   	    \draw[thick] (1/2,1/2) -- (0,1/2);
	   	    \draw[thick] (1/2,1) -- (1,1/2);
	   	    \draw[thick] (1/2,1) -- (0,1/2);
	   	    \draw[thick] (1/2,0) -- (0,1/2);
	   	    \draw[thick] (1/2,0) -- (1,1/2);
		\end{tikzpicture} \hspace{5em}		\begin{tikzpicture}[scale=0.9]
\draw[thick] (0,0) -- (1,0);
\draw[thick] (0,0) -- (1/2,1.732/2);
\draw[thick] (1/2,1.732/2) -- (1,0);
\draw[thick] (0,0) -- (1/2,-1.732/2);
\draw[thick] (1,0) -- (1/2,-1.732/2);
\draw[thick] (1,0) -- (3/2,1.732/2);
\draw[thick] (2,0) -- (3/2,1.732/2);
\draw[thick] (2,0) -- (1,0);
\draw[thick] (1/2,1.732/2) -- (3/2,1.732/2);
\draw[thick] (1/2,-1.732/2) -- (3/2,-1.732/2);
\draw[thick] (1,0) -- (3/2,-1.732/2);
\draw[thick] (2,0) -- (3/2,-1.732/2);
		\end{tikzpicture}
\caption{Type I (criss-cross) and Type II singularities (three-line-cross).}
\label{fig:singularities}
\end{figure}
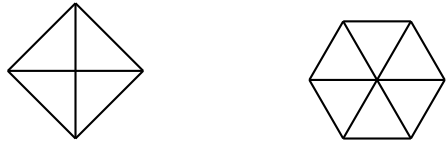

 In fact, the space $\div \Sigma_k^0$ can be explicitly characterized. In the absence of nearly singular vertices (defined precisely later),  $(\Sigma_k^0, \div \Sigma_k^0)$ constitutes a stable pair, see \Cref{thm:lagrange}. In practice, several approaches can generate triangular meshes avoiding singular vertices. For example, a general triangular mesh can be refined using the Morgan-Scott split, while a general rectangular mesh can be transformed into a fish-bone mesh, as illustrated in \Cref{fig:fg-fb}.

\begin{figure}[htbp]
\begin{tikzpicture}[scale=1.8, thick]
	
	\coordinate (A) at (0,1);
	\coordinate (B) at (-0.6,0);
	\coordinate (C) at (0.6,0);
	
	\def\r{0.25} 
	
	\coordinate (a) at ($(B)!0.5!(C)!\r!(A)$); 
	\coordinate (b) at ($(C)!0.5!(A)!\r!(B)$); 
	\coordinate (c) at ($(A)!0.5!(B)!\r!(C)$); 
	
	\coordinate (bc1) at (b);
	\coordinate (bc2) at (c);
	
	\coordinate (ab1) at ($(B)!0.5!(C)!\r!(A)$);
	\coordinate (ab2) at ($(C)!0.5!(B)!\r!(A)$);
	
	\draw (A) -- (B) -- (C) -- cycle;
	
	\draw[dashed] (a) -- (b) -- (c) -- cycle;
	
	\draw (A) -- (b) -- (c) -- cycle;
	\draw (a) -- (ab1) -- (ab2) -- cycle;
	
	\draw (B) -- (c) -- (a);
	\draw (C) -- (b) -- (a);
	\draw (B) -- (ab1);
	\draw (C) -- (ab2);
	
	\draw (c) -- (ab1);
	\draw (b) -- (ab2);
	
	
\end{tikzpicture} \hspace{5em}		\begin{tikzpicture}[scale=1.8, thick]
	
	\coordinate (O) at (0,0);       
	\coordinate (A) at (0,1);       
	\coordinate (B) at (1,1);       
	\coordinate (C) at (1,0);       
	\coordinate (M) at (0.5,0);     
	\coordinate (N) at (0.5,1);     
	\coordinate (L) at (0,0.5);     
	\coordinate (R) at (1,0.5);     
	\coordinate (Center) at (0.5,0.5); 
	
	\draw (O) -- (C) -- (B) -- (A) -- cycle;
	\draw (L) -- (R);
	\draw (M) -- (N);
	
	\draw (0,0.5) -- (0.5,1);
	\draw (0,0) -- (0.5,0.5);
	\draw (0.5,1) -- (1,0.5);
	\draw (0.5,0.5) -- (1,0);
	
	
\end{tikzpicture}
\caption{Morgan--Scott split and fish-bone mesh.}
\label{fig:fg-fb}
\end{figure}
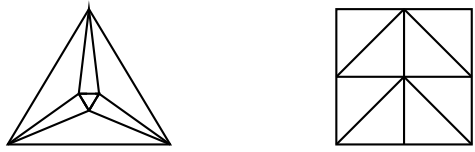

On the other hand, we are interested in finding the finite element pair satisfying the unconditional stability. Here ``unconditional'' indicates that the inf-sup condition can be established without any specific geometric structures on triangulations (in this case, singular vertices), and only depends on the shape regularity.  Motivated by the Falk--Neilan Stokes pair \cite{FalkNeilan13}, CG-DG pairs with extra vertex continuity are considered. 
In \Cref{thm:c2}, we show that the pair $\Sigma_k^2 \times Q_{k-1}^1$ achieves unconditional stability. Here $\Sigma_k^2 := \{ \sigma \in \Sigma_k^0, \sigma \text{ is } C^2 \text{ at vertices}\}$, and  $Q_{k-1}^1 := \{ \sigma \in Q_k^{-1}, \sigma \text{ is } C^1 \text{ at vertices}\}$. The extra smoothness at vertices cannot be reduced \Cref{thm:hermite}. The intuition of the unconditional stability comes from the fact that $C^4$ continuity is required for constructing $H^3$ conforming finite elements \cite{HuLinWuYuan25}.


\textbf{Relationship to the Stokes complex.} The results presented above share some similarities with respect to the divergence-free Stokes pair. Correspondingly, \eqref{eq:complex-sobolev} has strong relationship with the following Stokes complex:
\begin{equation}
\mathbf{P}_0 ~\hookrightarrow~ H^2(\Omega) \xrightarrow{\curl} H^1(\Omega;\mathbb{R}^2) \xrightarrow{\div} L^2(\Omega)\to 0.
\end{equation}

However, the motivations of the study are slightly different from each other. 
Unlike the Stokes pair, where weak divergence-free conditions suffice, the stress gradient model requires exact divergence surjectivity due to the non-coercive bilinear form. As a consequence, the surjectivity and the complex structure yield much importance for smoothed stress complexes. 

We now compare the results from the divergence-free Stokes pair and the results in this paper. From the comparison, we can see the algebraic patterns. The CG-DG pair in the Stokes equation is first constructed and analyzed by Scott and Vogelius \cite{ScottVogelius85}, and a recent re-investigation can be found in \cite{GuzmanScott19}. They showed that the CG-DG pair is stable if there is no criss-cross internal vertices and the polynomial degree $k \ge 4$.

Building on the discrete Stokes complex \cite{GuzmanNeilan14,FalkNeilan13,Neilan15,HuZhangZhang22}, stable pairs with divergence-free properties can be constructed. Specifically, the Hermite element (globally continuous with $C^1$ vertex continuity) can be chosen as the discrete velocity space, paired with a globally discontinuous space with $C^0$ continuity for the discrete pressure space. However, we can show that this is insufficent for \eqref{lsg-weak}. In fact, we can prove that for such a Hermite pair in \eqref{lsg-weak}, the divergence operator is onto if and only if there is no criss-cross singular vertices. A comprehensive comparison of the results are presented in \Cref{tab:rank-def}. 

\begin{figure}[htbp]
	\begin{tabular}{c|c|c|c|c}
	 &
		 Stokes & Stress Gradient & Stokes & Stress Gradient
		\\
		\hline

		& Type I & Type I & Type II & Type II
		\\
		
		&
		
		\begin{tikzpicture}[scale=1.8]
			\draw[thick] (1/2,1/2) -- (1/2,1);
	   	    \draw[thick] (1/2,1/2) -- (1,1/2);
	   	    \draw[thick] (1/2,1/2) -- (1/2,0);
	   	    \draw[thick] (1/2,1/2) -- (0,1/2);
	   	    \draw[thick] (1/2,1) -- (1,1/2);
	   	    \draw[thick] (1/2,1) -- (0,1/2);
	   	    \draw[thick] (1/2,0) -- (0,1/2);
	   	    \draw[thick] (1/2,0) -- (1,1/2);
		\end{tikzpicture}
		&
		\begin{tikzpicture}[scale=1.8]
			\draw[thick] (1/2,1/2) -- (1/2,1);
\draw[thick] (1/2,1/2) -- (1,1/2);
\draw[thick] (1/2,1/2) -- (1/2,0);
\draw[thick] (1/2,1/2) -- (0,1/2);
\draw[thick] (1/2,1) -- (1,1/2);
\draw[thick] (1/2,1) -- (0,1/2);
\draw[thick] (1/2,0) -- (0,1/2);
\draw[thick] (1/2,0) -- (1,1/2);
		\end{tikzpicture}
		&
		\begin{tikzpicture}[scale=0.9]
\draw[thick] (0,0) -- (1,0);
\draw[thick] (0,0) -- (1/2,1.732/2);
\draw[thick] (1/2,1.732/2) -- (1,0);
\draw[thick] (0,0) -- (1/2,-1.732/2);
\draw[thick] (1,0) -- (1/2,-1.732/2);
\draw[thick] (1,0) -- (3/2,1.732/2);
\draw[thick] (2,0) -- (3/2,1.732/2);
\draw[thick] (2,0) -- (1,0);
\draw[thick] (1/2,1.732/2) -- (3/2,1.732/2);
\draw[thick] (1/2,-1.732/2) -- (3/2,-1.732/2);
\draw[thick] (1,0) -- (3/2,-1.732/2);
\draw[thick] (2,0) -- (3/2,-1.732/2);
		\end{tikzpicture}&		\begin{tikzpicture}[scale=0.9]
		\draw[thick] (0,0) -- (1,0);
		\draw[thick] (0,0) -- (1/2,1.732/2);
		\draw[thick] (1/2,1.732/2) -- (1,0);
		\draw[thick] (0,0) -- (1/2,-1.732/2);
		\draw[thick] (1,0) -- (1/2,-1.732/2);
		\draw[thick] (1,0) -- (3/2,1.732/2);
		\draw[thick] (2,0) -- (3/2,1.732/2);
		\draw[thick] (2,0) -- (1,0);
		\draw[thick] (1/2,1.732/2) -- (3/2,1.732/2);
		\draw[thick] (1/2,-1.732/2) -- (3/2,-1.732/2);
		\draw[thick] (1,0) -- (3/2,-1.732/2);
		\draw[thick] (2,0) -- (3/2,-1.732/2);
	\end{tikzpicture}\\
	\hline $(\Sigma_k^0, Q_{k-1}^{-1})$
   & 1 &  3 & 0 & 1 \\
   \hline
   $(\Sigma_k^1, Q_{k-1}^0)$
   &0 & 1 &0 &0 \\
\hline
$(\Sigma_k^2, Q_{k-1}^1)$
& 0  &0 &0  & 0
	\end{tabular}
	\caption{The rank deficiency of Stokes pair and stress gradient pair.}
    \label{tab:rank-def}
\end{figure}

\textbf{Notations.} In this paper, let $\mathcal T$ be a conforming planar triangulation of a polygon domain $\Omega$. Moreover, we assume that $\Omega$ is contractible.  Let $\mathcal E$ and $\mathcal V$ be the set of all edges and vertices, respectively. We will also use $T$, $e$, $v$ to denote a face (element), edge, and vertex.

 Define the rigid body motion space $\mathbf{RM}=\mathrm{span}\{(1,0)^T,~(0,1)^T,~(-y,x)^T\}$, which is the kernel of the deformation tensor. We define the perpendicular operator for a vector $\begin{bmatrix} a \\ b \end{bmatrix}$ as $\begin{bmatrix} a \\ b \end{bmatrix}^{\perp} := \begin{bmatrix} b \\ -a \end{bmatrix}$. 

\textbf{Organization.} The paper is organized as follows. In \Cref{sec:continuous}, we discuss the well-posedness and a priori estimates of \eqref{lsg-weak}. In \Cref{sec:discrete}, we show the main results of the construction of the stable pairs, and the corresponding error estimates are shown in \Cref{sec:error}. In \Cref{sec:numerics}, we show the numerical results. Some concluding remarks and the discussions are shown in \Cref{sec:discussion}.

\section{ Wellposedness and  a priori estimates for continuous problems}
\label{sec:continuous}

In this section, we show the well-posedness of \eqref{lsg-weak} and some parameter-robust estimates. For $\iota > 0$, let
$a_{\iota}(\sigma,\tau):=\iota^2(\nabla  \mathcal{A} \sigma,\nabla  \tau)+( \mathcal{A}\sigma, \tau)$ and $b(\sigma,u)=(\mathrm{div}\sigma,u).$ Therefore, the mixed form can be reformulated as 
\begin{equation}
\begin{cases}
    a_{\iota}(\sigma,\tau)+b(\tau,u)=0,& \forall\tau\in H^1(\Omega; \mathbb S^2),\\
   b(\sigma,v)=(f,v), &\forall v\in L^2(\Omega; \mathbb R^2).
\end{cases}
\end{equation}

For the stress space $\Sigma := H^1(\Omega;\mathbb S^2)$, we define the $\iota$-norm as $\| \sigma\|_{\iota}:=\iota| \sigma|_{H^1}+\|\mathrm{div} \sigma\|_{L^2}+\| \sigma\|_{L^2}$. For the displacement space $Q:= L^2(\Omega;\mathbb R^2)$, we equip it with the standard $L^2$ norm. The goal of this section is to derive some estimates independent of $\lambda$ and $\iota$.

To verify the wellposedness, we need to show that $a_{\iota}$ is coercive in the kernel space $Z:=\{  \sigma \in H^1(\Omega;\mathbb{S}^{2})~:~\mathrm{div} \sigma = 0\}$. We recall some results from the standard analysis of the linear elasticity equations. Hereafter, the hidden constants of $\lesssim$ are independent of $\lambda$ and $\iota$. 

\begin{lemma}
For $\tau\in H^1(\Omega;\mathbb{S}^2)$, the estimates hold pointwise:
\begin{equation}
    \nabla\tau^D:\nabla\tau^D +\mathrm{div}\tau\cdot\mathrm{div}\tau\geq \frac{1}{8} \nabla\tau:\nabla\tau.
\end{equation}    
\label{coerecivity-lemma1}
\end{lemma}
\begin{proof}

Denote by $\|v\|_0 = v\cdot v$ and $|v|_1 = \nabla v : \nabla v$ as the pointwise $L^2$ norm and $H^1$ semi-norm. 
   By direct computation, we have 
{
\begin{equation}
    \nabla\tau^D:\nabla\tau^D = \frac{1}{2}|\tau_{11} - \tau_{22}|_{1}^2 + 2|\tau_{12}|_{1}^2 = \frac{1}{2} |\tau_{11}|_{1}^2 + \frac{1}{2} |\tau_{22}|_{1}^2 + 2|\tau_{12}|_{1}^2 + \nabla \tau_{11} \cdot \nabla \tau_{22}.
\end{equation}

Since we have 
$$ \tau_{11,x} \tau_{22,x} + \tau_{11,y} \tau_{22,y} = (\tau_{11,x} + \tau_{12,y}) \tau_{22,x} + \tau_{11,y}(\tau_{12,x} + \tau_{22,y}) - \tau_{12,y}\tau_{22,x} - \tau_{12,x} \tau_{11,y}.$$
Therefore, 
$$\nabla \tau_{11}\cdot\nabla \tau_{22} \ge - \|\div \tau\|_{0}^2 - \frac{1}{8} |\tau_{22}|_{1}^2 - \frac{1}{8} |\tau_{11}|_{1}^2  - 2|\tau_{12}|_{1}^2 - \frac{1}{4}|\tau_{11}|_{1}^2 - \frac{1}{4}|\tau_{22}|_{1}^2.$$

Therefore, we finish the proof.
}

\end{proof}
A corollary of \Cref{coerecivity-lemma1} is
$$
\|\nabla\tau\|^2_{L^2}\leq 8\|\nabla\tau^D\|^2_{L^2}+8\|\div\tau\|^2_{L^2}.
$$


We now introduce the following estimates for symmetric matrix-valued function $\tau$:
\begin{equation} 
\mathcal{A}\tau:\tau= (\frac{1}{2\mu}\tau^D:\tau^D+\frac{1}{4\lambda+4\mu}\mathrm{tr}(\tau)\mathrm{tr}(\tau)),~\forall \tau\in \mathbb{S}^2,
\label{Corecivity-elasticity-cont.}
\end{equation}
where $\tau^D=\tau-\frac{1}{2}\mathrm{tr}\tau\mathrm{I}$ is the traceless part of $\tau$.

\begin{proposition}[Coercivity]
For all $\tau\in Z$ (which implies that $\div \tau = 0$) and $\int_{\Omega}\mathrm{tr}(\tau)=0,$
\begin{equation} 
\iota^2(\nabla  \mathcal{A} \tau,\nabla  \tau)+( \mathcal{A}\tau, \tau) \gtrsim \|\tau\|^2_{\iota}.\
\label{coercivity-cont.}
\end{equation}
\end{proposition}
\begin{proof}

The following result can be found in \cite[Proposition 9.1.1]{MixedFEMbook}: 
    For  $\tau\in L^2(\Omega; \mathbb S^2)$ and 
    $\int_{\Omega}\mathrm{tr}(\tau)=0$, it holds that  
    $$
    \|\tau\|_{L^2}\lesssim (\|\tau^D\|_{L^2}+\|\div \tau\|_{L^2}).
    $$
    Here the hidden constant only depends on $\Omega$. Together with \eqref{Corecivity-elasticity-cont.}, we have 
  $$
  \|\tau\|_{L^2}\lesssim \|\tau^D\|_{L^2}\leq 2\mu(\mathcal{A}\tau,\tau),
  $$
  when $\tau\in Z$ and $\int_{\Omega}\mathrm{tr}(\tau)=0$. Applying gradient operator to \eqref{Corecivity-elasticity-cont.}, we obtain 
  $$
  \|\frac{1}{8}\nabla\tau\|_{L^2}\leq \|\nabla\tau^D\|_{L^2}\leq 2\mu(\nabla\mathcal{A}\tau,\nabla\tau),~\forall \tau\in Z.
  $$
Thus we finish the proof.  
\end{proof}

The following lemma gives the inf-sup condition of $b(\sigma,v)$, which is a standard result for linear elasticity equations \cite{MixedFEMbook}.
\begin{lemma}[Inf-sup conditions]
    For any $v\in L^2(\Omega;\mathbb R^2)$, there exists a $\tau\in H^1(\Omega;\mathbb S^2)$ such that
    $$
    \mathrm{div}\tau = v,~\|\tau\|_{\iota}\lesssim  \|v\|_{L^2}.
    $$
    \label{inf-sup-cont.}
\end{lemma}
\begin{proof} 
For all $v\in L^2(\Omega;\mathbb S^2)$, there exists a exact solution $u\in H^2(\Omega)\cap H_0^1(\Omega)$, such that
$$
-\div(2\mu \operatorname{sym} \operatorname{grad} (u)+\lambda\div uI)= v, 
$$
and $\|u\|_{H^2}+\lambda\|\div u\|_{L^2}\leq \|v\|_{L^2}$.
Let $$\tau = 2\mu \operatorname{sym} \operatorname{grad} (u)+\lambda\div u I\in \Sigma.$$ Since  $\|\tau\|_{\iota}\leq 2\|\tau\|_{H^1}$ we finish the proof.
\end{proof}
By the Ladyzhenskaya-Babu\v{s}ka-Brezzi theory, we establish the wellposedness of problem \eqref{lsg-weak}.
\begin{theorem}
    For all $f\in L^2(\Omega;\mathbb R^2)$, there exists a unique solution pair $(\sigma_{\iota},u_{\iota})\in H^1(\Omega;\mathbb S^2)\times L^2(\Omega;\mathbb R^2)$ of \eqref{lsg-weak}, such that
    $$
    \|\sigma_{\iota}\|_{\iota}+\|u_{\iota}\|_{L^2}\lesssim \|f\|_{L^2}.
    $$
    Here the hidden constant is independent of $\iota$ and $\lambda$.
\end{theorem}

\begin{proof}
Since $\|\sigma\|_{\iota}\geq \|\div\sigma\|_{L^2}$, the boundedness property of $b(\sigma,v)$ is trivial. By inf-sup stability, there exists unique $\sigma^{\perp}\in Z^{\perp}$, the orthogonal complement of $Z$ in $H^1(\Omega;\mathbb S^2)$ with respect to $\iota$ norm, such that
$$
b(\sigma^{\perp},v)=(f,v),~\forall v\in Q, \text{ and } \|\sigma^{\perp}\|_{\iota}\lesssim \|f\|_{L^2}.
$$
By \eqref{Corecivity-elasticity-cont.}, and $\|\mathrm{tr}(\sigma)\|_{L^2}\leq 2\|\sigma\|_{L^2},$ we have  
$$
a_{\iota}(\sigma,\tau)\leq (\frac{1}{2\mu}+\frac{1}{(\lambda+\mu)})(\|\sigma\|_{L^2}\|\tau\|_{L^2}+\iota^2\|\nabla\sigma\|_{L^2}\|\nabla\tau\|_{L^2}),~\forall~\tau,\sigma\in \Sigma.
$$
By Lax-Milgram Theorem, there exists $\sigma^0\in Z$, such that
$$
a_{\iota}(\sigma^0,\tau_z)=-a_{\iota}(\sigma^{\perp},\tau_z),~\forall~\tau_z\in Z,
$$
and  $\|\sigma^0\|_{\iota}\lesssim \|\sigma^{\perp}\|_{\iota}.$ Then $\sigma_{\iota} = \sigma^{\perp}+\sigma^0$ is the unique solution of \eqref{lsg-weak}. Finally, by the inf-sup condition again and the first equation of \eqref{lsg-weak}, there exists $u_{\iota}\in L^2(\Omega;\mathbb R^2)$, such that
$$
b(\tau,u_{\iota})=-a_{\iota}(\sigma_{\iota},\tau),~\forall~\tau\in Z^{\perp},
$$
and $\|u_{\iota}\|_{L^2}\lesssim\|\sigma_{\iota}\|_{\iota}.$ Thus we finish the proof.
\end{proof}

Now we derive the asymptotic result of the solution. Let $(\sigma_{\iota}, u_{\iota})$ solves \eqref{lsg-weak}, and $(u_0,\sigma_0)$ solves the following equation:
\begin{equation}
\begin{cases}
    (\mathcal A\sigma_0,\tau)+b(\tau,u_0)=0,& \forall\tau\in H(\div;\Omega),\\
   b(\sigma_0,v)=(f,v), &\forall v\in L^2(\Omega).
\end{cases}
\label{eq:iota0}
\end{equation}
\begin{proposition}
\label{prop:bdlayer-type}
Assume $\sigma_0\in H^2(\Omega;\mathbb{S}^2),~u_0 \in H^2(\Omega;\mathbb{R}^2),$ then
we have $\|\sigma_{\iota} - \sigma_0\|_{\iota} + \|u_{\iota} - u_0\|_{L^2} = O(\iota^{3/2}).$
\end{proposition}
\begin{proof}
    Clearly, for all $\tau \in H^1(\Omega;\mathbb S^2)$ and $v \in L^2(\Omega;\mathbb R^2)$ we have 
\begin{equation}
\begin{cases}
    \iota^2(\nabla \mathcal A( \sigma_{\iota} - \sigma_0),\nabla\tau) + (\mathcal A( \sigma_{\iota} - \sigma_0), \tau) +b(\tau,u_{\iota} - u_0)= - \iota^2(\nabla \mathcal A\sigma_0, \nabla \tau)\\
   b(\sigma_{\iota} - \sigma_0,v)=0.
\end{cases}
\label{equation:bdlayer-diff}
\end{equation}
Thus, we have $\div(\sigma_{\iota} - \sigma_0) = 0.$ Let $\tau = \sigma_{\iota} - \sigma_0,$ we then have 
$$\iota^2(\nabla \mathcal A(\sigma_\iota - \sigma_0), \nabla(\sigma_{\iota} - \sigma_0)) + (\mathcal A(\sigma_{\iota} - \sigma_0),\sigma_{\iota} - \sigma_0) = -\iota^2 (\nabla \mathcal A \sigma_0, \nabla(\sigma_{\iota} - \sigma_0)). $$
This implies 
\begin{equation}\begin{split} 
\iota^2|\sigma_{\iota} - \sigma_0|_{H^1}^2 + \|\sigma_{\iota} - \sigma_0\|_{L^2}^2 \lesssim&  -\iota^2 (\nabla \mathcal A \sigma_0,\nabla( \sigma_{\iota} - \sigma_0)) 
\\ = & \iota^2(\Delta \mathcal A \sigma_0, \sigma_{\iota} - \sigma_0)  -  \iota^2 \int_{\partial \Omega} \partial_{n} (\mathcal A\sigma_0):(\sigma_{\iota} - \sigma_0) \\ 
\lesssim & \iota^2 \|\partial_n \sigma_0\|_{L^2(\partial \Omega)} \|\sigma_{\iota} - \sigma_0 \|_{L^2(\partial \Omega)} + \iota^2 |\sigma_0|_{H^2}\|\sigma_{\iota} - \sigma_0\|_{L^2}.
\end{split}
\end{equation}
By trace theorem, we then have 
$$\iota^2|\sigma_{\iota} - \sigma_0|_{H^1}^2 + \|\sigma_{\iota} - \sigma_0\|_{L^2}^2 \lesssim \iota^2 \|\partial_n \sigma_0\|_{L^2(\partial \Omega)}\|\sigma_{\iota} - \sigma_0\|_{H^1}^{1/2}\|\sigma_{\iota} - \sigma_0\|_{L^2}^{1/2} + \iota^2 |\sigma_0|_{H^2}\|\sigma_{\iota} - \sigma_0\|_{L^2}.$$
By Cauchy-Schwarz inequality, we then have 
$$ \iota^2|\sigma_{\iota} - \sigma_0|_{H^1}^2 + \|\sigma_{\iota} - \sigma_0\|_{L^2}^2 \lesssim \iota^3\|\partial_n \sigma_0\|_{L^2(\partial \Omega)}^2 + \iota^4|\sigma_0|_{H^2}^2,$$
which is equivalent to 
$$\iota|\sigma_{\iota} - \sigma_0|_{H^1} + \|\sigma_{\iota} - \sigma_0\|_{L^2} \lesssim \iota^{3/2} \|\partial_n \sigma_0\|_{L^2(\partial \Omega)} + \iota^2 |\sigma_0|_{H^2}.$$

Next, we give the estimates on the term $u_{\iota} - u_0.$ Take $\xi \in H^1(\Omega; \mathbb S^2)$ such that $\|\xi\|_{H^1} \lesssim \|u_{\iota} - u\|_{L^2}$ and $\div \xi = u_{\iota} - u_0$. Therefore, we have 
$$\iota^2(\nabla \mathcal A(\sigma_\iota - \sigma_0), \nabla \xi) + (\mathcal A(\sigma_{\iota} - \sigma_0),\xi) + \|u_{\iota} - u_0\|_{L^2} = -\iota^2 (\nabla \mathcal A \sigma_0, \nabla \xi). $$
By Cauchy-Schwarz inequality, we then have
$$\|u_{\iota} - u_0\|_{L^2}^2 \lesssim \iota^2 |\sigma_{\iota}|_{H^1} \|\xi\|_{H^1}  + \|\sigma_{\iota} - \sigma_{0}\|_{L^2} \|\xi\|_{L^2},$$
which implies 
$$  \|u_{\iota} - u_0\|_{L^2} \lesssim \iota^2 |\sigma_{\iota}|_{H^1} + \|\sigma_{\iota} - \sigma_{0}\|_{L^2} \lesssim \iota^{3/2}\|\partial_n \sigma_0\|_{L^2{\partial \Omega}} + \iota^2|\sigma_0|_{H^1} + \iota^3|\sigma_0|_{H^2}.$$

\end{proof}

\section{Stable Finite Element Pairs for stress gradient elasticity}
\label{sec:discrete}
This section constructs a family of stable pairs for stress gradient elasticity. Specifically, we seek a stress-displacement pair $(\Sigma_h, Q_h) \subseteq H^1(\Omega; \mathbb{S}^2) \times L^2(\Omega; \mathbb{R}^2)$ such that the divergence operator $\div : \Sigma_h \to Q_h$ is surjective and admits a bounded (generalized) inverse. That is, for each $q_h \in Q_h$, there exists $\sigma_h \in \Sigma_h$ satisfying $\div \sigma_h = q_h$ and $\|\sigma_h\|_{H^1} \lesssim \|q_h\|_{L^2}$. For such a pair, we can consider the discrete problem
\begin{equation}
\begin{cases}
    a_{\iota}(\sigma_h,\tau_h)+b(\tau_h,u_h)=0,& \forall\tau_h\in\Sigma_h,\\
   b(\sigma_h,v_h)=(f,v_h), &\forall v_h\in Q_h.
\end{cases}
\label{lsg-discrete}
\end{equation}
The wellposedness and the error estimates of the discrete problem are discussed in \Cref{sec:error}. 


In this section, we introduce the CG-DG pairs as follows. Define the following stress spaces
$$\Sigma^0_{k} = \{  \sigma \in C^0(\mathcal T;\mathbb S^2) : \sigma|_T \in P_{k}(T;\mathbb S^2),  \sigma \in C^0(\mathcal V)\},$$
$$\Sigma^1_{k} = \{ \sigma \in C^0(\mathcal T;\mathbb S^2) : \sigma|_T \in P_{k}(T;\mathbb S^2),  \sigma \in C^1(\mathcal V)\},$$
$$\Sigma^2_{k} = \{ \sigma \in C^0(\mathcal T;\mathbb S^2) : \sigma|_T \in P_{k}(T;\mathbb S^2), \sigma \in C^2(\mathcal V)\}.$$
Here, $\sigma \in C^0(\mathcal V)$ means that $\sigma$ is continuous at vertices, $\sigma \in C^{r}(\mathcal V)$ means that $\nabla^{r'} \sigma$ is continuous at vertices for $0 \le r' \le r.$

Accordingly, we define the following displacement space.  
$$Q^{-1}_{k-1} = \{ q \in L^2(\mathcal T;\mathbb R^2) : q|_T \in P_{k-1}(T;\mathbb R^2)\},$$
$$Q^0_{k-1} = \{ q\in L^2(\mathcal T;\mathbb R^2) :  q|_T \in P_{k-1}(T;\mathbb R^2), q \in C^0(\mathcal V)\},$$
$$Q^1_{k-1} = \{ q \in L^2(\mathcal T;\mathbb R^2) :  q|_T \in P_{k-1}(T;\mathbb R^2),q\in C^1(\mathcal V)\}.$$

To address the rank deficiency and the stability issue caused by the geometric structure, we now formally define two types of singularities.

\begin{definition}[Type I Singularity]
A vertex $v$ is said to have a Type I singularity if and only if it is an interior vertex and all edges incident to $v$ lie on at most two distinct lines.
\end{definition}

Such singularities have been thoroughly investigated in the Stokes' pair \cite{ScottVogelius85,GuzmanScott19}, and the construction of the $C^1$ spline space on general triangular meshes \cite{MorganScott75}.

The second type of singularity is the three-line-cross singularity (Type II), which is not fully addressed in the literature.

\begin{definition}[Type II Singularity]
A vertex $v$ is said to have a Type II singularity if and only if it is an interior vertex and all edges incident to $v$ lie on exactly three distinct lines. A singular vertex is called \emph{non-degenerate} if it is surrounded by six elements (see \Cref{fig:typeII-non}); otherwise, it is called \emph{degenerate} (see \Cref{fig:typeII-de}).
\end{definition}

\begin{figure}[htbp]
\begin{tikzpicture}[scale=1.5]
		\draw[thick] (0,0) -- (1,0);
		\draw[thick] (0,0) -- (1/2,1.732/2);
		\draw[thick] (1/2,1.732/2) -- (1,0);
		\draw[thick] (0,0) -- (1/2,-1.732/2);
		\draw[thick] (1,0) -- (1/2,-1.732/2);
		\draw[thick] (1,0) -- (3/2,1.732/2);
		\draw[thick] (2,0) -- (3/2,1.732/2);
		\draw[thick] (2,0) -- (1,0);
		\draw[thick] (1/2,1.732/2) -- (3/2,1.732/2);
		\draw[thick] (1/2,-1.732/2) -- (3/2,-1.732/2);
		\draw[thick] (1,0) -- (3/2,-1.732/2);
		\draw[thick] (2,0) -- (3/2,-1.732/2);
	\end{tikzpicture}
	\caption{Non-degenerate Type II singularity.}
    \label{fig:typeII-non}
\end{figure}
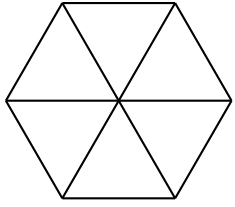

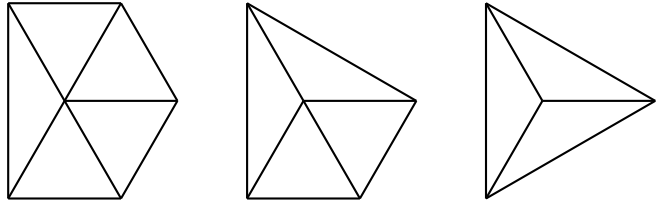
\begin{figure}[htbp]
		\begin{tikzpicture}[scale=1.5]
\draw[thick] (1/2,1.732/2) -- (1,0);
\draw[thick] (1/2,1.732/2) -- (1/2,-1.732/2);
\draw[thick] (1,0) -- (1/2,-1.732/2);
\draw[thick] (1,0) -- (3/2,1.732/2);
\draw[thick] (2,0) -- (3/2,1.732/2);
\draw[thick] (2,0) -- (1,0);
\draw[thick] (1/2,1.732/2) -- (3/2,1.732/2);
\draw[thick] (1/2,-1.732/2) -- (3/2,-1.732/2);
\draw[thick] (1,0) -- (3/2,-1.732/2);
\draw[thick] (2,0) -- (3/2,-1.732/2);
		\end{tikzpicture}
		\qquad 
				\begin{tikzpicture}[scale=1.5]
\draw[thick] (1/2,1.732/2) -- (1,0);
\draw[thick] (1/2,1.732/2) -- (1/2,-1.732/2);
\draw[thick] (1,0) -- (1/2,-1.732/2);
\draw[thick] (2,0) -- (1,0);
\draw[thick] (2,0) -- (1/2,1.732/2);
\draw[thick] (1/2,-1.732/2) -- (3/2,-1.732/2);
\draw[thick] (1,0) -- (3/2,-1.732/2);
\draw[thick] (2,0) -- (3/2,-1.732/2);
		\end{tikzpicture}
\qquad 
				\begin{tikzpicture}[scale=1.5]
\draw[thick] (1/2,1.732/2) -- (1,0);
\draw[thick] (1/2,1.732/2) -- (1/2,-1.732/2);
\draw[thick] (1,0) -- (1/2,-1.732/2);
\draw[thick] (2,0) -- (1,0);
\draw[thick] (2,0) -- (1/2,1.732/2);
\draw[thick] (2,0) -- (1/2, -1.732/2);
		\end{tikzpicture}
\caption{Degenerate Type II singularity. The numbers of the surrounding element of a Type II singular vertex can be 5,4,3. We denote them as Type II-1, II-2, and II-3, accordingly.}
\label{fig:typeII-de}
	\end{figure}

In this section, we establish a key result on the surjectivity of the divergence operator for the stress-displacement pairs proposed above in stress gradient elasticity. The main theorem presented below characterizes the precise geometric conditions under which the divergence operator $\div : \Sigma_h \to Q_h$ is surjective, ensuring the existence of a generalized inverse operator. 
%
\begin{theorem}
\label{thm:overall}
Suppose that $k \ge 7$. Then the following statements hold:
\begin{enumerate}
	\item (Unconditional Stability) For the pair $(\Sigma_k^2, Q_{k-1}^1)$, the divergence operator $\div : \Sigma_k^2 \to Q_{k-1}^1$ is surjective for any triangulation $\mathcal{T}$. 
	\item For the Hermite pair $(\Sigma_k^1, Q_{k-1}^0)$, the divergence operator $\div : \Sigma_k^1 \to Q_{k-1}^0$ is surjective for any triangulation $\mathcal{T}$ that has no Type I singularities.
	\item For the Lagrange pair $(\Sigma_k^0, Q_{k-1}^{-1})$, the divergence operator $\div : \Sigma_k^0 \to Q_{k-1}^{-1}$ is surjective for any triangulation $\mathcal{T}$ that has no Type I or Type II singularities.
\end{enumerate}
For a general triangulation $\mathcal{T}$, the following quantitative results hold:
\begin{enumerate}
	\item[(2')] The rank deficiency of the pair $(\Sigma_k^1, Q_{k-1}^0)$ is given by \begin{equation}\dim Q_{k-1}^0 - \dim \div \Sigma_k^1 = |\mathcal{V}_I|.\end{equation}
	\item[(3')] The rank deficiency of the pair $(\Sigma_k^0, Q_{k-1}^{-1})$ is given by \begin{equation}\dim Q_{k-1}^{-1} - \dim \div \Sigma_k^0 = |\mathcal{V}_{II}| + 3  |\mathcal{V}_I|.\end{equation}
\end{enumerate}
Here, $\mathcal{V}_I$ denotes the set of vertices with Type I singularities, and $\mathcal{V}_{II}$ denotes the set of vertices with Type II singularities. See \Cref{tab:rank-def} for a summary and comparison with the Stokes pair.
\end{theorem}
\begin{proof}
(1) comes from \Cref{thm:c2}.  (2) and (2') come from \Cref{thm:hermite}. (3) and  (3') come from  \Cref{thm:lagrange}.
\end{proof}

\begin{remark}
The polynomial degree requirement $k \ge 7$ is imposed for technical reasons. In this paper, our primary focus is to classify singular vertices, while determining the optimal polynomial degree requirement is not our main objective at present and will be addressed in future work. Especially, reducing the degrees of freedom may require additional conditions, see \cite{GuzmanScott18}.
\end{remark}

\begin{remark}
We make a comparison of the proposed pairs with the Stokes pairs. The pairs of (2) and (2') can be regarded as matrix-valued versions of the Falk--Neilan Stokes pair, while the results from items (3) and  (3') can be regarded as matrix-valued versions of the Scott--Vogelius Stokes pair. This can help the readers compare the results better. 
\end{remark}

\subsection*{Sketch of Proofs}
The proof of the theorem above relies on a detailed characterization of $\div \Sigma_k^r$. In the following, we use three subsections to provide this characterization and establish the stability results for the pair $(\Sigma_k^r, \div \Sigma_k^r)$ for $r = 2, 1, 0$, respectively.

The theorem will be obtained item by item, and for clarity we show the sketch of the overall proof here. To establish unconditional stability, we construct the discrete smoothed stress complexes, starting from $C^2$ scalar element spaces. Based on the bubble complexes and the finite element complexes, we can derive the unconditional stability of the global function space pair $(\Sigma_k^2, Q_k^1)$. To remove the extra continuity at vertices, we consider the function data on the local vertex patches, and make suitable modifications to proceed. The geometric constraints appear during the modification: When dealing with $\Sigma_k^1$, only type I singularities appear; when dealing with $\Sigma_k^0$, both type I and type II singularities appear.

Therefore, the proof relies on analyzing data on the vertex patches, calling for more notations. For a vertex $z \in \mathcal{V}$, denote by $\st(z)$ the patch of $z$, namely, it consists of elements surrounding $z$. The elements $T_i$ in $\st(z)$ are numbered counterclockwise as $T_1, T_2, \dots, T_m$, where $m$ is the total number of elements in the patch. We use $y_s$ to denote the vertex shared by elements $T_s$ and $T_{s+1}$, and $t_s$ ($n_s := t_s^{\perp}$, resp.) to represent the unit tangent (normal, resp.) vector of the directed edge $\overrightarrow{zy_s}$. When the vertex is interior, the subscription is periodically labeled. See \Cref{fig:star-z} for an illustration. Additionally, we define $h_s$ as the height with respect to vertex $y_s$. Specifically, we have 
\begin{equation}
h_s := 
\begin{cases}
	|\overline{zy_s}| \sin \theta_s & \text{ in } T_{s}, \\
	|\overline{zy_s}| \sin \theta_{s+1} & \text{ in } T_{s+1}.
\end{cases}
\end{equation}
When focusing on a single element $T_s$, we may denote $t_-$ as $t_s$ and $t_+$ as $t_{s+1}$ if the context permits.

For each vertex $z \in \mathcal{V}$, we define the nodal basis function as 
\[
\psi_z(v) = \begin{cases} 
1 & \text{if } v = z, \\
0 & \text{if } v \neq z.
\end{cases}
\]
Within each element $T$, the restriction of $\psi_z$ to $T$ is the barycentric coordinate function associated with vertex $z$. Similarly, we use $n_{\pm}$ and $\psi_{\pm}$ when the context permits.

\begin{figure}[htbp]
		\begin{tikzpicture}[scale=1.5]
		\draw[dashed] (0,0)--(1,1.8);
		\draw  (1,1.8)--(3,0);
		\draw   (0,0)--(3,0);
\draw   (1.5,-1)--(1.5,0);
\draw   (1.5,0)--(1,1.8);
	\draw   (1.5,-1)--(3,0);
    \draw   (1.5,-1)--(0,0);

		\fill (0,0) circle (1.5pt); 
		\fill (1,1.8) circle (1.5pt); 
		\fill (3,0) circle (1.5pt); 
          \fill (1.5,0) circle (1.5pt);
\fill (1.5,-1) circle (1.5pt);
		\node [right=] at (3,0) {$y_s$};
        \node [right=] at (1.5,-0.2) {$z$};
        \node [right=] at (1.7,-0.5) {$T_s$};
        \node [right=] at (1.7,0.5) {$T_{s+1}$};
        \node [right=] at (0.8,0.5) {};
        \node [right=] at (0.8,-0.5) {$T_{s-1}$};
            \node [below=] at (2.3,0) {$t_{s}$};
              \node [left=] at (1.8,0.8) {$t_{s+1}$};
                \node [below=] at (1.5,-1) {$y_{s-1}$};
                  \node [left=] at (1.6,-0.3) {$t_{s-1}$};
	\end{tikzpicture}
    \caption{An illustration for the notations on the patch $z$.}
    \label{fig:star-z}
	
\end{figure}
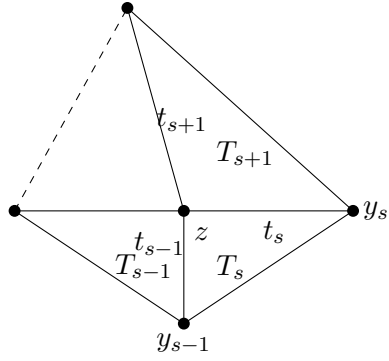

Using the notations introduced above, we can define a measure to quantify the extent to which a triangulation is singular or nearly singular with respect to Type I or Type II singularities. 
For any interior vertex $x \in \mathcal{V}$, we define
\begin{equation}
\label{eq:theta_I}
\Theta_I(x) := \max \{ |\sin(\theta_i + \theta_{i+1})| \},
\end{equation}
where the maximum is taken over all $i$ in the patch $\st x$. It can be verified that $\Theta_I(x) = 0$ if and only if $x \in \mathcal{V}_I$. Similarly, we define
\begin{equation}
\label{eq:theta_II}
\Theta_{II}(x) := \max \{ \min ( |\sin(\theta_i + \theta_{i+1} + \theta_{i+2})|, |\sin(\theta_i + \theta_{i+1})|, |\sin(\theta_{i+1} + \theta_{i+2})| ) \},
\end{equation}
for any interior vertex $x \in \mathcal{V}$, where the maximum is taken over all $i$ in the patch $\st x$. It can be verified that $\Theta_{II}(x) = 0$ if and only if $x \in \mathcal{V}_I$ or $x \in \mathcal{V}_{II}$.  

In this section, the hidden constants in $\lesssim$ might depend on the region $\Omega$, the polynomial degree $k$, and the shape regularity constant, but independent with the triangulation $\mathcal T$. The dependency on singularities in triangulations will be given explicitly via $\Theta_I(x)$ or $\Theta_{II}(x)$. 

\begin{remark}
The quantities mentioned above are defined only for the interior vertices. For the boundary vertices, they do not contribute to the cokernels of divergence operators.
\end{remark}

\subsection{The Pair $\Sigma_k^2 \times Q_{k-1}^1$: Unconditional Stability and Discrete Complexes} 

In this subsection, we show that the pair $\Sigma_k^2 \times Q_{k-1}^1$ is unconditionally stable with respect to both types of singularities. The main theorem is summarized as follows.
\begin{theorem}
\label{thm:c2}
The divergence operator $\div : \Sigma_k^2 \to Q_{k-1}^1$ is surjective for any mesh $\mathcal T$, provided that $k \ge 7$. Moreover, if $\mathcal T$ is shape regular, then for each $q  \in Q_{k-1}^1$, there exists $\sigma \in \Sigma_k^2$ such that $\div \sigma = q $ and $\|\sigma \|_{H^1} \lesssim \|q\|_{L^2}.$ 
\end{theorem}

\Cref{thm:c2} establishes that unconditional stability is achieved by imposing $C^2$ continuity at the vertices in the stress element space $\Sigma_k^2$ and $C^1$ continuity at the vertices in the displacement element space $Q_{k-1}^1$.

This theorem arises from a finite element discretization of the smoothed stress complex \eqref{eq:complex-sobolev}. To elucidate this, we recall the $H^3$ conforming finite element $U_{k+2}^4$ from \cite{Zenisek74,HuLinWu24}, defined for $k \ge 7$. The local shape function space of $U_{k+2}^4$ is the polynomial space $P_{k+2}$. For $u \in P_{k+2}$ on an element $T$, the degrees of freedom are given as follows:
\begin{enumerate}
	\item $\nabla^{\alpha} u(v)$ for each vertex $v$ with $|\alpha| \le 4$;
	\item $\int_e u \, p $ for $p \in P_{k-8}(e)$, $\int_e \frac{\partial u}{\partial n} p $ for $p \in P_{k-7}(e)$, and $\int_e \frac{\partial^2 u}{\partial n^2} p \, ds$ for $p \in P_{k-6}(e)$, for each edge $e$ of $T$;
	\item $\int_T u \, p $ for $p \in P_{k-7}(T)$.
\end{enumerate}
It can be readily checked that the finite element $U_{k+2}^4$ is unisolvent, and the resulting space is $H^3$ conforming.

The local shape function space of $\Sigma_k^2$ is $P_{k}(T;\mathbb S^2)$, and for $\sigma \in P_{k}(T;\mathbb S^2)$, the degrees of freedom are given as follows: 
\begin{enumerate}
    \item $\nabla^{\alpha} \sigma(v)$ for each vertex $v$ of $T$ and $|\alpha| \le 2$;
    \item $\int_e \sigma : p $ for $p \in P_{k-6}(e;\mathbb S^2)$ for each edge $e$ of $T$; 
    \item $\int_ T \sigma : p $ for $p \in P_{k-3}^{(0)}(T;\mathbb S^2).$ 
     Here $P_{k-3}^{(0)}(T) := \{ p \in P_{k-3}(T) : p(v) = 0,\,\, \forall v \in \mathcal V(T)\}.$
\end{enumerate}

Then, we can construct the following sequence: 
\begin{equation}
\label{eq:complex-fe}
\mathbf{P}_1~ {\hookrightarrow}~ U^4_{k+2}(\mathcal T) \xrightarrow{\curl\curl^T} \Sigma_{k}^2 (\mathcal T) \xrightarrow{\div} Q_{k-1}^1(\mathcal T) \to 0.
\end{equation}

The exactness of \eqref{eq:complex-fe} can be checked from direct calculation.
\begin{proposition}
\label{prop:complex-fe}
The finite element sequence \eqref{eq:complex-fe} forms a complex. When $\Omega$ is topologically trivial, this sequence is exact.
\end{proposition}

\begin{proof}
For $\sigma \in \Sigma_k^2(\mathcal T) \subseteq H^1(\Omega;\mathbb S^2)$ with $\div \sigma = 0$, there exists $\psi \in H^3$ such that $\psi$ is piecewise polynomial. Since $\psi \in H^3$, we know that $\psi$ and $\nabla \psi$ are continuous at vertices. Therefore, $\psi \in U_{k+2}^4(\mathcal T)$. To establish exactness, it suffices to verify the dimension count for the sequence \eqref{eq:complex-fe}. By direct calculation, we obtain the following dimensions:
\begin{align*}
\dim U_{k+2}^4(\mathcal{T}) &= \frac{1}{2}(k-6)(k-5) |\mathcal{T}| + (3k-18) |\mathcal{E}| + 15 |\mathcal{V}|, \\
\dim \Sigma_k^2(\mathcal{T}) &= \left( \frac{3}{2}(k+2)(k+1) - 9(k+1) \right) |\mathcal{T}| + (3k-15) |\mathcal{E}| + 18 |\mathcal{V}|, \\
\dim Q_{k-1}^1(\mathcal{T}) &= (k(k+1) - 18) |\mathcal{T}| + 6 |\mathcal{V}|.
\end{align*}
Since
\[
\dim U_{k+2}^4(\mathcal{T}) - \dim \Sigma_k^2(\mathcal{T}) + \dim Q_{k-1}^1(\mathcal{T}) = 3 (|\mathcal{T}| - |\mathcal{E}| + |\mathcal{V}|) = 3,
\]
where the final equality holds for a topologically trivial domain $\Omega$ by the Euler equation, it then follows that the finite element sequence is exact.
\end{proof}

To establish the stability results, we begin by proving the following results, which provides the exactness of the bubble stress complex. We define the local bubble spaces $\mathring{U}_{k+2}(T)$, $\mathring{\Sigma}_k(T)$, and $\mathring{Q}_{k-1}(T)$ on each element $T \in \mathcal{T}$ as follows respectively:
\begin{align*}
\mathring{U}_{k+2}(T) &:= \left\{ u \in P_{k+2}(T) : \nabla^{\alpha} u(v) = 0, \, \forall |\alpha| \le 4, \, u|_{\partial T} = \frac{\partial u}{\partial n}\big|_{\partial T} = \frac{\partial^2 u}{\partial n^2}\big|_{\partial T} = 0 \right\}, \\
\mathring{\Sigma}_k(T) &:= \left\{ \sigma \in P_k(T; \mathbb{S}^2) : \nabla^{\alpha} \sigma(v) = 0, \, \forall |\alpha| \le 2, \, \sigma|_{\partial T} = 0 \right\}, \\
\mathring{Q}_{k-1}(T) &:= \left\{ q \in P_{k-1}(T; \mathbb{R}^2) : \nabla^{\alpha} q(v) = 0, \, \forall |\alpha| \le 1 \right\}.
\end{align*}

\begin{proposition}
\label{prop:bubble-complex}
The following sequence on each element $T \in \mathcal{T}$:
\begin{equation}
\label{eq:bubble-complex}
0 \to \mathring{U}_{k+2}(T) \xrightarrow{\curl \curl^T} \mathring{\Sigma}_k(T) \xrightarrow{\div} \mathring{Q}_{k-1}(T) \cap \mathbf{RM}^{\perp} \to 0,
\end{equation}
forms a complex and is exact. {Here $\cap \mathbf{RM}^{\perp}$ means that the $\mathbf{RM}$ moment is zero.}
\end{proposition}

\begin{proof}
It is straightforward to see \eqref{eq:bubble-complex} is a complex. For exactness, it suffices to verify the dimension count. We compute the dimensions of the local bubble spaces as follows:
\begin{align*}
\dim \mathring{U}_{k+2}(T) &= \dim P_{k-7}(T) = \frac{1}{2}(k-6)(k-5), \\
\dim \mathring{\Sigma}_k(T) &= 3 \dim P_k(T) - 9(k-5) - 54 = \frac{3}{2}(k+2)(k+1) - 9(k-5) - 54, \\
\dim \mathring{Q}_{k-1}(T) &= 2 \dim P_{k-1}(T) - 12 = k(k+1) - 18.
\end{align*}
It follows that
\[
\dim \mathring{U}_{k+2}(T) - \dim \mathring{\Sigma}_k(T) + \dim \mathring{Q}_{k-1}(T) = 3 = \dim \mathbf {RM}.
\]
Therefore, we conclude the exactness.
\end{proof}

By the surjectivity result and scaling argument, we have the bubble stability.
 
\begin{corollary}[Bubble Stability]
\label{cor:bubble-stable}
Suppose that $k \ge 7$. For any $q \in P_{k-1}(T; \mathbb{R}^2)$ satisfying:
\begin{enumerate}
	\item $q(v) = \nabla q(v) = 0$ for all vertices $v \in \mathcal{V}(T)$, and
	\item $\int_T q \cdot p \, dx = 0$ for all $p \in \mathbf{RM}$,
\end{enumerate}
there exists $\sigma \in \mathring{\Sigma}_k(T)$ such that $\div \sigma = q$ and
$
\|\sigma\|_{H^1} \lesssim \|q\|_{L^2},
$ where the hidden constant depends only on the shape regularity constant of $T$.
\end{corollary}

Now the proof of \Cref{thm:c2} follows a standard argument \cite{FalkNeilan13,HuZhang14}, with the help of \Cref{cor:bubble-stable}. 
\begin{proof}[Proof of \Cref{thm:c2}]

Given any $q_h\in Q_{k-1}^1 \subset L^2(\Omega;\mathbb R^2)$, it suffices to show that there exists $\tau\in H^1(\Omega;\mathbb S^2)$ such that $\mathrm{div}\tau=q_h.$ The proof will be divided into two steps. 

In the first step, we select $\sigma^{(1)} \in \Sigma_k^2$ such that

\begin{enumerate}
    \item For all vertices $v$, $\sigma_{i,j}^{(1)}(v) = 0,  \partial_{x} \sigma_{i,j}^{(1)}(v)  = \partial_{x} \sigma_{i,j}^{(1)}(v) = \frac{1}{2}q_h(v),
$ and 
$ 
\partial_{xx} \sigma_{ij}^{(1)}(v)  =  \partial_{x} q_h(v), \partial_{yy} \sigma_{i,j}^{(1)}(v) =\partial_{y} q_h(v), \partial_{xy} \sigma_{i,j}^{(1)}(v) = 0.
$
\item For all $e\in\mathcal{E}$ and $p \in P_{k-6}(T; \mathbb S^2)$,  
$
\int_e\sigma^{(1)} n_e\cdot p=\int_e\tau_h n_e\cdot p.
$
Here $n_e$ is the unit normal vector of edge $e$. 
\item 
For all $T\in\mathcal{T}$, for all $p\in P_{k-3}^0(T;\mathbb{S}^2)$, 
$
\int_T \sigma^{(1)} : p= \int_T \tau_h : p.
$
\end{enumerate}

It can be checked that at each vertex $v \in \mathcal V$, we have 
$
\mathrm{div}\sigma^{(1)}(v)=q_h(v),~\nabla\mathrm{div}\sigma^{(1)}(v) = \nabla q_h(v).
$
Moreover, for $k\geq 7$, it holds that
$$
\int_T\mathrm{div}\sigma^{(1)}\cdot p=\int_{\partial T}\sigma^{(1)}n_{\partial T} \cdot p=\int_{\partial T}\tau n_{\partial T}\cdot p=\int_T\mathrm{div}\tau\cdot p=\int_T q_h\cdot p,~p\in \mathbf{RM}.
$$
This comes from the fact that $p \cdot n_{e}$ is linear function for all $p \in \mathbf{RM}.$

In the second step, we notice that in each $T$, it holds that $q_h - \div \sigma^{(1)} \in \mathring{Q}_{k-1}(T) \cap \mathbf{RM}^{\perp}.$ 
By \Cref{cor:bubble-stable}, there exist $\sigma^{(2)}\in \Sigma,~\sigma^{(2)}|_T\in \mathring{\Sigma}_{k}(T) $, such that
$$
\mathrm{div}\sigma^{(2)} = q_h-\mathrm{div}\sigma^{(1)},\text{ and }\|\sigma^{2}\|_{H^1}\lesssim \| q_h-\div \sigma^{(1)}\|_{ L^2}. 
$$
Let $\sigma=\sigma^{(1)}+\sigma^{(2)}$ we finish the construction.

To complete the proof, it suffices to estimate $\|\sigma\|_{H^1}$. Since $\|\sigma\|_{H^1} \le \|\sigma^{(1)} \|_{H^1} + \|\sigma^{(2)}\|_{H^1} \lesssim \|\sigma^{(1)} \|_{H^1} + \|q_h\|_{L^2}$, it then suffices to estimate $\|\sigma^{(1)}\|_{H^1}.$

Let $\tau_h \in \Sigma_k$ be the Scott-Zhang interpolant of $\tau$.
Then $(\sigma^{(1)}-\tau_h)|_T\in \mathcal{P}_k(T;\mathbb{S}^2)$. For $\Sigma_k^2$, let $\varphi_{v,0,i}$, $\varphi_{v,1,i}$, $\varphi_{v,2,i}$ be the nodal basis functions associated with function  value $(\cdot)(v) $ , gradient value $\nabla (\cdot)(v) $ and Hessian value $\nabla^2 (\cdot)(v) $, respectively. $\varphi_{e,j}$ is the nodal basis function corresponding to degrees of freedom $\int_e (\cdot) p_j$ and $\varphi_{T,j}$ for $\int_T(\cdot)p_j.$  
Expend $(\sigma^{(1)}-\tau_h)$ by the set of basis function of $\Sigma(T)$, we then have
\begin{align*}
    \|\sigma^{(1)}-\tau_h\|_{H^1(T)}&=\|\sum_{i=0}^{2}\sum_{v\in\mathcal{V}(T)}\nabla^i(\sigma^{(1)}-\tau_h)(v)\varphi_{v,i}+\sum_j\sum_{e\in\mathcal{E}(T)}\int_{e}(\sigma^{(1)}-\tau_h)p_j \varphi_{e,j} \\
    &+\sum_j\int_T(\sigma^{(1)}-\tau_h)p_j\varphi_{T,j}  \|_{H^1(T)}.
\end{align*}
By scaling argument, 
\begin{align*} 
&\|\sum_j \sum_{v\in\mathcal{V}(T)}\nabla^i(\sigma^{(1)}-\tau_h):p_j (v)\varphi_{v,i,j}\|^2_{H^1(T)}\lesssim \sum_{a\in \mathcal{V}(T)}(h_T^i|\nabla^i\tau_h(a)|^2+h_T^2|q_h(a)|^2+h_T^4|\nabla q_h(a)|^2);\\
&\|\sum_j\sum_{e\in\mathcal{E}(T)}\int_{e}(\sigma^{(1)}-\tau_h)q_j \varphi_{e,j}\|_{H^1(T)}=\|\sum_j\sum_{e\in\mathcal{E}(T)}\int_{e}(\tau-\tau_h)q_j \varphi_{e,j}\|_{H^1(T)}\lesssim h_T^{-1/2}\|\tau-\tau_h\|_{L^2(\partial T)};\\
&\|\sum_j\int_T(\sigma^{(1)}-\tau_h)q_j \varphi_{T,j}  \|_{H^1(T)}
\lesssim \|\sum_j\int_T(\tau-\tau_h)q_j\varphi_{T,j}  \|_{H^1(T)}
\lesssim h_T^{-1}\|\tau-\tau_h\|_{L^2(T)}.
\end{align*} 

By applying a scaling argument and the inverse inequality, we obtain
\[
\|\sigma^{(1)} - \tau_h\|_{H^1({T})} \lesssim \|\tau\|_{H^1(\widetilde{T})} + \|q_h\|_{L^2(\widetilde{T})},
\]
where $\tilde{T}$ denotes the union of all elements adjacent to $T$. Consequently, by shape regularity we have 
\[
\|\sigma^{(1)}\|_{H^1} \lesssim \|\tau_h\|_{H^1} + \|\tau\|_{H^1} + \|q_h\|_{L^2} \lesssim \|q_h\|_{L^2}.
\]

\end{proof}

\subsection{The Pair $\Sigma_k^1 \times Q_{k-1}^0$: Matrix-version Falk-Neilan pair} 

In this subsection, we analyze the results for the pair $\Sigma_k^1 \times Q_{k-1}^0$, provided that $k \geq 7$. First, we see how the rank deficiency arises for singular vertices. Recall that for each Type I singular vertex $z$, we denote the faces in the patch surrounding $z$ by $T_1, T_2, T_3, T_4$ in counterclockwise order. 


Now we fix one of its surrounding elements $T$. Recall the conventional notations in \Cref{fig:star-z}, $n_+$ represents the outer normal vector, while $n_-$  represents the inner normal vector. Note that See \Cref{fig:triangle} for an illustration. We then define \begin{equation} \label{eq:Jzt} J_{z,T} = t_{+}n_{-}  ^T + t_{-} n_{+} ^T.\end{equation}

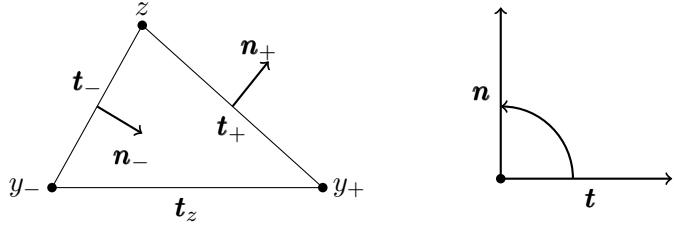
\begin{figure}[htbp]
\centering
	\begin{minipage}{0.4\textwidth}
	\begin{tikzpicture}[scale=1.2]
		\draw (0,0)--(1,1.8);
\node [above] at (0.4,0.9) {$\pmb t_{-}$};
\draw  (1,1.8)--(3,0);
\draw   (0,0)--(3,0);
\node [below] at (1.5,0) {$\pmb t_{z}$};
\node [below] at (2,0.9) {$\pmb t_{+}$};
\fill (0,0) circle (1.5pt); 
\fill (1,1.8) circle (1.5pt); 
\fill (3,0) circle (1.5pt); 
\node [above] at (1,1.8) {$z$};
\node [left] at (0,0) {$y_-$};
\node [right=] at (3,0) {$y_+$};
\draw[thick,->] (0.5,0.9) -- (1.0,0.6);
\node[left] at (1.2,0.3) {$\pmb n_{-}$};
\draw[thick,->] (2,0.9) -- (2.4,1.4);
\node[above] at (2.3,1.3) {$\pmb n_{+}$};
	\end{tikzpicture}
\end{minipage}
	\begin{minipage}{0.4\textwidth}
	\begin{tikzpicture}[scale=1.2]
\draw[->,thick] (-0.5,0) -- (1.4,0); 
\draw[->,thick] (-0.5,0) -- (-0.5,1.9); 
\node [below] at (0.5,0) {$\pmb t$};
\node [left] at (-0.5,0.95) {$\pmb n$};
\draw[thick,->] (0.3,0) arc[start angle=0, end angle=90, radius=0.8];
\fill (-0.5,0) circle (1.5pt);
	\end{tikzpicture}
\end{minipage}
	\caption{Notations in an element.}
	\label{fig:triangle}
\end{figure}

\begin{lemma}[Necessary condition for Type I singularity]
\label{lem:necessary-type-I}
Let $z$ be a Type I singular vertex, with four elements $T_1, T_2, T_3, T_4$ sharing $z$, as shown in Figure~\ref{fig:Type I}. Then for any continuous, piecewise $C^2$, symmetric matrix-valued function $\tau \in C^0(\st  (z) ;\mathbb S^2)$, it holds that
\begin{equation}\label{eq:type I}
J_{z,T_1} : \nabla \div \tau |_{T_1}(z) + J_{z,T_2} : \nabla \div \tau |_{T_2}(z) + J_{z,T_3} : \nabla \div \tau |_{T_3}(z) + J_{z,T_4} : \nabla \div \tau |_{T_4}(z) = 0.
\end{equation}
\end{lemma}

\begin{figure}[htbp]
		\begin{tikzpicture}[scale=4]
			\draw[thick] (1/2,1/2) -- (1/2 + 1/4,1);
\draw[thick] (1/2,1/2) -- (1,1/2);
\draw[thick] (1/2,1/2) -- (1/2 - 1/4,0);
\draw[thick] (1/2,1/2) -- (0,1/2);
\draw[thick] (1/2 + 1/4,1) -- (1,1/2);
\draw[thick] (1/2 + 1/4,1) -- (0,1/2);
\draw[thick] (1/2 - 1/4,0) -- (0,1/2);
\draw[thick] (1/2 - 1/4,0) -- (1,1/2);
		\node [right=] at (1/2+1/8,1/2+1/8) {$T_1$};
			\node [right=] at (1/2-1/8,1/2+1/8) {$T_2$};
				\node [right=] at (1/2-1/4,1/2-1/8) {$T_3$};
				\node [right=] at (1/2+1/16,1/2-1/8) {$T_4$};
				\fill (1/2,1/2) circle (0.5pt);
                \node[below=] at (1/2,1/2) {$z$};
                \node[right=] at (1,1/2) {$y_1$};
                \node[above=] at (1/2+1/4,1) {$y_2$};
                \node[left=] at (0,1/2) {$y_3$};
                \node[below=] at (1/2-1/4,0) {$y_4$};
		\end{tikzpicture}
\caption{A configuration for type I singularity.}
\label{fig:Type I}
		\end{figure}
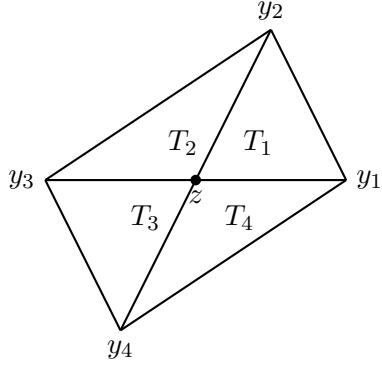
\begin{proof}
Without loss of generality, we rotate the edge $\overline{zy_1}$ to align with the $x$-axis, and set $z$ as the origin. In this case, $\alpha = 0$ and we define $\theta$ as the angle between $\overline{zy_2}$ and the $x$-axis. Consider the transformation:
\[
\begin{pmatrix}
\tilde{x} \\
\tilde{y}
\end{pmatrix}
=
\begin{pmatrix}
1 & 0 \\
\cos\theta & \sin\theta
\end{pmatrix}
\begin{pmatrix}
x \\
y
\end{pmatrix}.
\]
By direct computation within a single element, we obtain the following results when evaluating at $z$:
\begin{align*}
\partial_x(\mathrm{div}\tau)_1 &= \partial^2_{\tilde{x}}\tau_{11} + 2\cos\theta\partial^2_{\tilde{x}\tilde{y}}\tau_{11} + \sin\theta\partial^2_{\tilde{x}\tilde{y}}\tau_{12} + \cos^2\theta\partial^2_{\tilde{y}}\tau_{11} + \cos\theta\sin\theta\partial^2_{\tilde{y}}\tau_{12}, \\ 
\partial_y(\mathrm{div}\tau)_1 &= \sin\theta\partial^2_{\tilde{x}\tilde{y}}\tau_{11} + \sin\theta\cos\theta\partial^2_{\tilde{y}}\tau_{11} + \sin^2\theta\partial^2_{\tilde{y}}\tau_{12}, \\
\partial_x(\mathrm{div}\tau)_2 &= \partial^2_{\tilde{x}}\tau_{12} + 2\cos\theta\partial^2_{\tilde{x}\tilde{y}}\tau_{12} + \sin\theta\partial^2_{\tilde{x}\tilde{y}}\tau_{22} + \cos^2\theta\partial^2_{\tilde{y}}\tau_{12} + \sin\theta\cos\theta\partial^2_{\tilde{y}}\tau_{22}, \\
\partial_y(\mathrm{div}\tau)_2 &= \sin\theta\partial^2_{\tilde{x}\tilde{y}}\tau_{12} + \sin\theta\cos\theta\partial^2_{\tilde{y}}\tau_{12} + \sin^2\theta\partial^2_{\tilde{y}}\tau_{22}.
\end{align*}
Eliminating the mixed derivatives, we find the following identity holds at $z$:
\begin{align*} 
J_{z,T_1} &: \nabla\mathrm{div}\tau|_{T_1}= \sin\theta\partial^2_{\tilde{x}}\tau_{11} - \sin\theta\cos^2\theta\partial^2_{\tilde{y}}\tau_{11} - 2\cos\theta\sin^2\theta\partial^2_{\tilde{y}}\tau_{12} - \sin^3\theta\partial^2_{\tilde{y}}\tau_{22}.
\end{align*}
Since $\tau$ is continuous, it follows that $\partial^2_{\tilde{x}}\tau$ is continuous on $\overline{zy_1}, \overline{zy_3}$ and $\partial^2_{\tilde{y}}\tau$ is continuous on $\overline{zy_2}, \overline{zy_4}$. Therefore, for the values at $z$ we have 
$$J_{z,T_1} : \nabla\mathrm{div}\tau|_{T_1} + J_{z,T_2} : \nabla\mathrm{div}\tau|_{T_2} + J_{z,T_3} : \nabla\mathrm{div}\tau|_{T_3} + J_{z,T_4} : \nabla\mathrm{div}\tau|_{T_4}  = 0.$$
\end{proof}

\begin{remark}
For the singular case, it holds that
$J_z := J_{z,1} = - J_{z,2} = J_{z,3} = - J_{z,4}.$ Therefore, \eqref{eq:type I} can be rewritten as an alternating sum: 
\[  J_z : \sum_{s = 1}^4 (-1)^s \nabla \div \tau(z)|_{T_s} = 0.\]
Suppose that $\tau \in \Sigma_k^0$. If we set $q = \div \tau$, then we have
\begin{equation}\label{eq:hm-cond}J_{z} :\nabla q_{1}(z)+ J_{z} :\nabla  q_{3}(z) = J_{z}: \nabla q_{2}(z)+ J_{z}: \nabla  q_{4}(z).\end{equation} Here, $q_s = q|_{T_s}.$
\end{remark}
%

Now we prove the sufficiency. We need to show that for all multi-valued vector-valued functions $q$ such that \eqref{eq:hm-cond} holds, there exists $\sigma \in \Sigma_k^0(z)$ such that $\div \sigma(z) = q(z)$. The challenge {arises} here from the nontrivial null space of the mapping $\tau \mapsto J_z : \tau$ in \eqref{eq:hm-cond}. To resolve this, we deal with two different {types} of $\tau$ separately: the first type contributes nothing to both terms in \eqref{eq:hm-cond}, and the second type has nonzero but equal contributions to both sides.

Define the null space $
\mathcal{N}_{z,T} := \{ \tau \in \mathbb{R}^{2 \times 2} : J_{z,T} : \tau = 0 \}. $ Now we fix an element $T$ and recall the notation in \Cref{fig:triangle}.  Let
\begin{equation}
\label{eq:varphi_inT}
\varphi_1 = \psi_z^2 \psi_+ \psi_- t_- t_-^T, \quad \varphi_2 = \psi_z^2 \psi_+ \psi_- t_+ t_+^T, \quad \varphi_3 = \psi_z^2 \psi_+ \psi_- (t_- t_+^T + t_+ t_-^T).
\end{equation}
Clearly, all $\varphi_i$ for $i = 1,2,3$ are supported in $T$, with $\nabla \varphi_i(v) = 0$ for all $v \in \mathcal{V}$ and $\nabla^2 \varphi_i(v) = 0$ for $v \neq z$. Consequently, $\varphi_i \in \Sigma_4^1(\st (z))$, and vanishes on the boundary. We now evaluate $\nabla \div \varphi_i(z)$ to see they form a basis of the null space.

\begin{lemma}
\label{lemma:NzT}
Fix $z, T$ and let $\varphi_i$ be defined as \eqref{eq:varphi_inT}. Then, $\nabla \div \varphi_i(z), i = 1,2,3$ form a basis of $\mathcal N_{z,T}= \{ \tau \in \mathbb R^{2\times 2} : J_{z,T} : \tau = 0 \}.$
\end{lemma}

\begin{proof}
Note that $\nabla \psi_{-} = \frac{1}{h_{-}} n_+, \nabla \psi_{+} = - \frac{1}{h_+} n_-$, and $t_+ \cdot n_- = - t_- \cdot n_+ = \sin \theta$, where $\theta$ is the angle between $t_+$ and $t_-$. 
By direct calculation, it holds that

\begin{equation}
\begin{aligned}
\nabla \div \varphi_1(z) & =  \frac{\sin\theta}{h_+h_-} n_-t_-^T, \\ 
\nabla \div \varphi_2(z) & = 	-\frac{\sin\theta}{h_+h_-} n_+ t_+^T ,\\ 
\nabla \div \varphi_3(z) & = 	- \frac{\sin\theta}{h_+h_-} (n_- t_+^T - n_+ t_-^T). \\ 
\end{aligned}
\end{equation}
Therefore, it can be checked that $\nabla \div \varphi_i$ form a basis.

\end{proof}

%
%



 We now proceed with the proof through a case-by-case analysis, starting with the nonsingular case. The proof strategy is similar to that of \cite{GuzmanScott19}. Suppose that $z\in \mathcal{V} \setminus \mathcal{V}_I$ is a nonsingular vertex.  For any multi-valued vector-valued function $q \in Q_{k-1}^0$, we now find $\sigma \in \Sigma_k^1$ such that $\div \sigma(z) = q(z)$. Furthermore, we obtain bounded estimates whenever $\Theta_I(x)$ is bounded below (no nearly singular vertices). The following lemma provides the details.

\begin{lemma}[Gradient Value Modification for the Nonsingular Case]
\label{lem:hm-nonsingular}
For $z \in \mathcal{V} \setminus \mathcal{V}_I$ and any $p \in Q_{k-1}^0$, there exists $\sigma \in \Sigma_k^1$ satisfying:
\begin{enumerate}
\item $\sigma$ is supported in $\st(z)$.
\item $\mathrm{div} \sigma(v) = 0$ for all $v \in \mathcal{V}$.
\item $\nabla \mathrm{div} \sigma(v) = 0$ for all $v \neq z$.
\item $\nabla \mathrm{div} \sigma|_T(z) = \nabla q|_T(z)$ for all $T \in \st(z)$.
\end{enumerate}
Moreover, the bound $\displaystyle \|\sigma\|_{H^1} \lesssim \frac{1}{\Theta_I(z)^2} \|q\|_{L^2(\st(z))}$ holds.
\end{lemma}
\begin{proof}
The construction proceeds in several steps, with the following strategy. We first construct $\tau$ such that $\nabla \div \tau - \nabla q$ lies in the null space $\mathcal{N}_{z,T}$ for each $T \in \st(z)$. Then, we apply \Cref{lemma:NzT} to adjust the values lying in the null space. For brevity, we denote $J_{z,s} = J_{z,T_s}$ and $\mathcal{N}_{z,s} = \mathcal{N}_{z,T_s}$, where 
$
J_{z,s} = t_s n_{s-1}^T + t_{s-1} n_s^T.
$ Recall that this agrees with \eqref{eq:Jzt}. 

\textbf{Step 1:} We first consider $T_s$ with $\sin(\theta_s + \theta_{s-1}) = \Theta_I(z) > 0$ and show how to modify the gradient value in $T_s$. 
For such $s$, we define $\sigma_s^+ = \psi_z^3 \psi_s^2 t_{s+1} t_{s+1}^T$. This $\sigma_s^+$ satisfies conditions (1), (2), and (3) of \Cref{lem:hm-nonsingular}. For condition (4), we compute the value separately:
\begin{itemize}
\item In $T_{s+1}$, we have
\[
\nabla \div \sigma_s^+|_{T_{s+1}}(z) = (t_{s+1} \cdot \nabla \lambda_{y_s}) \nabla \lambda_{y_s} t_{s+1}^T = 0.
\]
\item In $T_s$, we have
\[
\nabla \div \sigma_s^+|_{T_s}(z) = (\nabla \lambda_{y_s} \cdot t_{s+1}) \nabla \lambda_{y_s} t_{s+1}^T = -\frac{\sin(\theta_s + \theta_{s+1})}{h_s^2} n_{s-1} t_{s+1}^T.
\]
\end{itemize}
Thus,
\begin{align*}
J_{z,s} : \nabla \div \sigma_s^+|_{T_s}(z) &= (t_s n_{s-1}^T + t_{s-1} n_s^T) : \left( -\frac{\sin(\theta_s + \theta_{s+1})}{h_s^2} n_{s-1} t_{s+1}^T \right) \\
&= \frac{\sin \theta_s}{h_s^2} \sin^2(\theta_s + \theta_{s+1}).
\end{align*}
Since $\sin(\theta_s + \theta_{s-1}) \neq 0$, we define
\[
\tau_s = \frac{h_s^2 (J_{z,s} : \nabla q(z))}{\sin \theta_s \sin^2(\theta_s + \theta_{s+1})} \sigma_s^+,
\]
ensuring $J_{z,s} : \nabla \div \tau_s(z) = J_{z,s} : \nabla q(z)$ in $T_s$. Moreover,
\[
\|\tau_s\|_{H^1} \lesssim \frac{h_s^2 |\nabla q(z)|}{\Theta_I(z)^2} \lesssim \frac{1}{\Theta_I(z)^2} \|q\|_{L^2(\st(z))}.
\]
Here the last inequality comes from the scaling argument.

\textbf{Step 2:} We now adjust the value for other $T_{s'}$, $s' \neq s$ to make sure the modified gradient is in the null space.
Without loss of generality, assume $s' = 1$. For each $r$, define $\sigma_r = \psi_z^3 \psi_r^2 t_r t_r^T$. In $T_r$,
\[
\nabla \div \sigma_r|_{T_r}(z) = (t_r \cdot \nabla \lambda_r) \nabla \lambda_r t_r^T = -\frac{\sin \theta_r}{h_r^2} n_{r-1} t_r^T,
\]
thus it holds that
\[
J_{z,r} : \nabla \div \sigma_r|_{T_r}(z) = (t_r n_{r-1}^T + t_{r-1} n_r^T) : \left( -\frac{\sin \theta_r}{h_r^2} n_{r-1} t_r^T \right) = -\frac{\sin^3 \theta_r}{h_r^2}.
\]
In $T_{r+1}$,
\[
\nabla \div \sigma_r|_{T_{r+1}}(z) = \frac{\sin \theta_r}{h_r^2} n_{r+1} t_r^T,
\]
and
\[
J_{z,r+1} : \nabla \div \sigma_r|_{T_{r+1}}(z) = \frac{\sin^3 \theta_{r+1}}{h_{r+1}^2}.
\]
Let $\tilde{\sigma}_r = \frac{h_r^2}{\sin^3 \theta_r} \sigma_r$. Then we have
\[
\|\tilde{\sigma}_r\|_{H^1} \lesssim h_r^2, \quad |J_{z,r+1} : \nabla \div \tilde{\sigma}_r|_{T_{r+1}}(z)| \lesssim 1.
\]
Define
\[
\tilde{\tau}_s = \alpha_{s-1} \frac{h_s^2 (J_{z,s} : \nabla q(z))}{\sin \theta_s \sin^2(\theta_s + \theta_{s+1})} \sigma_s^+,
\]
where $\sigma_s^+$ is defined in Step 1. 
Now let
\[
\tau_1 = \alpha_1 \tilde{\sigma}_1 + \cdots + \alpha_{s-1} \tilde{\sigma}_{s-1} + \tilde{\tau}_s,
\]
where $\alpha_1 = J_{z,1} : \nabla q(z)$, and $\alpha_r = \alpha_{r-1} (J_{z,r} : \nabla \div \tilde{\sigma}_{r-1}|_{T_r}(z))$ for $r = 2, \ldots, s-1$, sequentially. 
Thus, $|\alpha_r| \lesssim |\alpha_1| \lesssim |\nabla q(z)|$, and
\[
\|\alpha_r \tilde{\sigma}_r\|_{H^1} \lesssim |\nabla q(z)| h_r^2 \lesssim \|q\|_{L^2(\st(z))}, \quad \|\tilde{\tau}_s\|_{H^1} \lesssim \frac{1}{\Theta_I(z)^2} \|q\|_{L^2(\st(z))}.
\]
This $\tau_1$ satisfies $J_{z,1} : \nabla \div \tau_1|_{T_1}(z) = J_{z,1} : \nabla q(z)$, while $J_{z,s} : \nabla \div \tau_1|_{T_s}(z) = 0$ for $s \neq 1$. Similarly, we can construct $\tau_{s'}$ for other $s' \neq 1$.

\textbf{Step 3:} After the previous two steps, there exists $\tau \in \Sigma_0^k$ such that $\nabla \div \tau - \nabla q$ to lie in $\mathcal{N}_{z,T}$ for all $T \in \st(z)$. Using \Cref{lem:hm-nonsingular}, we conclude the proof.
\end{proof}

\begin{remark}
The above argument assumes that $z$ is an interior vertex. For the case that $z$ is a boundary vertex, we can choose any boundary element $T_s$ in Step 1 and then continue the remaining proof. 	
\end{remark}

Now we focus on the singular case. The idea behind the proof is similar to the non-singular case. 

\begin{lemma}[Gradient Value Modification for the Singular Case]
\label{lem:hm-singular}
For $z \in \mathcal V_I$, and any $q \in Q_{k-1}^0$ such that 
\eqref{eq:hm-cond} holds. 
There exists $\sigma \in \Sigma_{k}^{1}$, such that the following condition holds:
\begin{enumerate}
\item $\sigma$ is supported in $\st(z)$. 
\item 	$\mathrm{div}\sigma(v)=0,\text{ for all } v\in \mathcal{V}$,
\item	$\nabla\mathrm{div}\sigma(v)=0,~\text{for all } v\neq z$.
\item	$\nabla\mathrm{div}\sigma|_T(z)=\nabla q|_T(z)$ for all $T \in \st(z)$.
\end{enumerate}
Moreover, it holds that $\|\sigma\|_{H^1} \lesssim \|q\|_{L^2(\st(z))}.$
\end{lemma}
\begin{proof}
The proof is similar to the second step of \Cref{lem:hm-nonsingular}. Note in this case we have $J_{z,r+1} : \nabla \div \tilde \sigma_r|_{T_{r+1}}(z) = 1$. Therefore, we finish the proof.   
\end{proof}

The following theorem gives the full characterization of $\div \Sigma_k^1$. 
\begin{theorem}
\label{thm:hermite}
The divergence of the space $\Sigma_k^1$ is characterized as
\[
\div \Sigma_k^1 = \left\{ q \in Q_{k-1}^0 : J_z : \nabla q_1(v) + J_z : \nabla q_3(v) = J_z : \nabla q_2(v) + J_z : \nabla q_4(v), \ \forall v \in \mathcal{V}_I \right\},
\]
where $q_i := q|_{T_i}$ denotes the restriction of $q$ to element $T_i$ around a singular vertex. Moreover, if $\Theta_I(x) \geq \theta_I > 0$ for all $x \in \mathcal{V} \setminus \mathcal{V}_I$, then for any $q \in \div \Sigma_k^1$, there exists $\sigma \in \Sigma_k^1$ such that
\[
\div \sigma = q \quad \text{and} \quad \|\sigma\|_{H^1} \lesssim \theta_I^{-2} \|q\|_{L^2}.
\]
\end{theorem}

\begin{proof}
By \Cref{lem:hm-nonsingular,lem:hm-singular}, for each vertex $z \in \mathcal{V}$, there exists $\sigma_z \in \Sigma_k^1$ such that:
\begin{itemize}
\item $\sigma_z$ is supported in $\st(z)$,
\item $\nabla \div \sigma_z(v) = 0$ for all $v \neq z$,
\item $\nabla \div \sigma_z|_T(z) = \nabla q|_T(z)$ for all $T \in \st (z)$,
\end{itemize}
with the bound
\[
\|\sigma_z\|_{H^1} \lesssim \theta_I^{-2} \|q\|_{L^2(\st(z))}.
\]
Define $\sigma^{(1)} = \sum_{z \in \mathcal{V}} \sigma_z$. Then, $\nabla \div \sigma^{(1)} - \nabla q$ vanishes at each vertex in $\mathcal{V}$. Since each $\sigma_z$ is supported in $\st(z)$, it follows that
\[
\|\sigma^{(1)}\|_{H^1} \lesssim \theta_I^{-2} \|q\|_{L^2}.
\]
Next, observe that $q^{(1)} := q - \div \sigma^{(1)} \in Q_{k-1}^1$. By the surjectivity of the divergence operator from $\Sigma_k^2$ to $Q_{k-1}^1$, there exists $\sigma^{(2)} \in \Sigma_k^2$ such that
\[
\div \sigma^{(2)} = q^{(1)} = q - \div \sigma^{(1)},
\]
with
\[
\|\sigma^{(2)}\|_{H^1} \lesssim \|q^{(1)}\|_{L^2} \leq \|q\|_{L^2} + \|\div \sigma^{(1)}\|_{L^2} \lesssim \theta_I^{-2} \|q\|_{L^2}.
\]
Finally, set $\sigma = \sigma^{(1)} + \sigma^{(2)}$. Then,
\[
\div \sigma = \div \sigma^{(1)} + \div \sigma^{(2)} = \div \sigma^{(1)} + (q - \div \sigma^{(1)}) = q,
\]
and
\[
\|\sigma\|_{H^1} \leq \|\sigma^{(1)}\|_{H^1} + \|\sigma^{(2)}\|_{H^1} \lesssim \theta_I^{-2} \|q\|_{L^2},
\]
completing the proof.
\end{proof}

\subsection{The pair $\Sigma_k^0 \times Q_{k-1}^{-1}$: Matrix-version Scott--Vogelius pair} 
\label{sec:Sigma0}

This subsection establishes the characterization and stability of the divergence image of the symmetric tensor Lagrange elements, which can be regarded as a matrix-version of the Scott--Vogelius pair. As previously noted, the primary challenge arises from the presence of two types of singularities, with Type II singularities being of particular interest in the stress gradient equation.

We first present the necessary results for both types of singularities. 
\begin{lemma}[Necessary Conditions for Type I Singularity]
\label{lem:typeI-singularity}
Suppose $\sigma$ is continuous and piecewise $C^2$, with $q = \div \sigma$. Then, at any vertex $v \in \mathcal{V}_I$, the following equations hold:
\begin{equation}
\label{eq:lag-cond-typeI}
q_1(v) + q_3(v) = q_2(v) + q_4(v),
\end{equation}
\begin{equation}
\label{eq:lag-cond-typeI-grad}
J_z : \nabla q_1(v) + J_z : \nabla q_3(v) = J_z : \nabla q_2(v) + J_z : \nabla q_4(v).
\end{equation}
\end{lemma}

\begin{proof}
It suffices to prove \eqref{eq:lag-cond-typeI}, as \eqref{eq:lag-cond-typeI-grad} is addressed in \eqref{eq:hm-cond}. The condition in \eqref{eq:lag-cond-typeI} mirrors the analysis of the Scott-Vogelius Stokes pair.  
The symmetric tensors $t_1 t_1^T$, $t_2 t_2^T$, and $t_1 t_2^T + t_2 t_1^T$ form a basis of $\mathbb{S}^2$. Thus, we can express
\[
\sigma = \alpha t_1 t_1^T + \beta t_2 t_2^T + \gamma (t_1 t_2^T + t_2 t_1^T),
\]
where $\alpha$, $\beta$, and $\gamma$ are continuous Lagrange functions. The divergence is
\[
\div \sigma|_{T_i}(z) = (\partial_{t_1} \alpha_i + \partial_{t_2} \gamma_i) t_1 + (\partial_{t_2} \beta_i + \partial_{t_1} \gamma_i) t_2,
\]
where $\alpha_i = \alpha|_{T_i}(z)$, and $\beta_i$, $\gamma_i$ are similarly defined. Since $\sigma$ is continuous across elements, evaluating at $v \in \mathcal{V}_I$ and summing over elements yields \eqref{eq:lag-cond-typeI}, completing the proof.
\end{proof}

Now we focus on the Type II singularity vertex. We first consider the non-degenerate case.  Let $T_1,\cdots, T_6$ be the faces, with vertices $y_1,\cdots, y_6$, and the angle $\theta_1, \cdots, \theta_6$. See \Cref{fig:Type II-singular} for an illustration.  
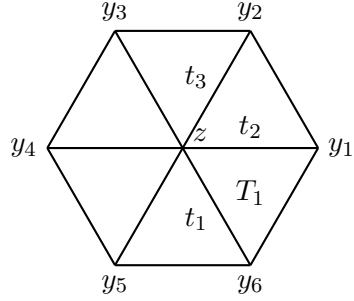
\begin{figure}[htbp]
\begin{tikzpicture}[scale=1.8]
\draw[thick] (0,0) -- (1,0);
\draw[thick] (0,0) -- (1/2,1.732/2);
\draw[thick] (1/2,1.732/2) -- (1,0);
\draw[thick] (0,0) -- (1/2,-1.732/2);
\draw[thick] (1,0) -- (1/2,-1.732/2);
\draw[thick] (1,0) -- (3/2,1.732/2);
\draw[thick] (2,0) -- (3/2,1.732/2);
\draw[thick] (2,0) -- (1,0);
\draw[thick] (1/2,1.732/2) -- (3/2,1.732/2);
\draw[thick] (1/2,-1.732/2) -- (3/2,-1.732/2);
\draw[thick] (1,0) -- (3/2,-1.732/2);
\draw[thick] (2,0) -- (3/2,-1.732/2);
 \node [right=] at (1,0.1) {$z$};
 \node [right=] at (2,0) {$y_1$};
\node [above=] at (1/2,1.732/2) {$y_3$};    
\node [above=] at (3/2,1.732/2) {$y_2$};
 \node [left=] at (0,0) {$y_4$};
 \node [below=] at (1/2,-1.732/2) {$y_5$};
 \node [below=] at (3/2,-1.732/2) {$y_6$};
 \node [below=] at (3/2,-1.732/2+0.7) {$T_1$};
  \node [below=] at (1.1,-1.732/2+0.5) {$t_1$};
  \node [above=] at (1.5,0) {$t_2$};
  \node [above=] at (1.1,+1.732/2-0.5) {$t_3$};
		\end{tikzpicture}
		\caption{An illustration for type II singularity, nondegenerate vertex.}
		\label{fig:Type II-singular}
\end{figure}

With these notations, we have the following lemma, showing the necessary conditions for type II singular vertices.

\begin{lemma}[Necessary Conditions for Type II Singularity]
\label{lem:necessary-type-II}
Let $z$ be a Type II singular vertex, with four elements $T_1, T_2, T_3, T_4, T_5, T_6$ sharing $z$, as shown in Figure~\ref{fig:Type II-singular}. Then for any continuous, piecewise $C^1$, symmetric matrix-valued function $\sigma \in C^0(\st (z);\mathbb S^2)$, it holds that
\begin{equation}
\label{eq:lag-cond-typeII}
\sum_{i=1}^6 (-1)^i \sin \theta_i \, \div \sigma_i(z) \cdot n_{i+1} = 0,
\end{equation}
where $\sigma_i = \sigma|_{T_i}$ is the restriction of $\sigma$ to element $T_i$.
\end{lemma}

The detailed proof of \Cref{lem:necessary-type-II} is similar to that of \Cref{lem:necessary-type-I}, and we show the proof in the appendix.

\begin{remark}
{
To see how \eqref{eq:lag-cond-typeII} is derived, we can let $\sigma$ be the Lagrange $P_1$ finite element space $\lambda_s \tau $, and $\tau$ be any piecewise constant symmetric matrix. We then have 
$$\div \sigma_s(z) = \frac{\tau \cdot n_{s-1}}{|zy_s| \sin \theta_s} \text { and } \div \sigma_{s+1}(z) = -\frac{\tau \cdot n_{s+1}}{|zy_s| \sin \theta_{s+1}}.$$
Since $n_{s-1}\cdot n_{s+1} = - n_{s+1} \cdot n_{s+2}$, we can see that any $\sigma = \lambda_s \tau $ satisfies \eqref{eq:lag-cond-typeII}. 
}
\end{remark}

\begin{remark}[Necessary case on degenerate case]
We make a remark here on the degenerate case. A straightforward observation is that to obtain the corresponding necessary condition, we can pretend the vertex is non-degenerate by elongating each line and then write down the corresponding equation \eqref{eq:lag-cond-typeII}. Therefore, we adopt the same notation as in the non-degenerate case. Namely, we use $T_{i}$, $t_i$, $n_i$ for $i = 1,2,\cdots, 6$ for the imaging elements and vectors, and use the notations $T_{s,s+1}$ to denote the actual elements. Therefore, we adhere to the unified condition \eqref{eq:lag-cond-typeII}. See \Cref{fig:Type II-singular} for an illustration. Note that this actually leaves lots of technical details to the verification of the sufficiency, see \Cref{sec:appendix}.

\begin{figure}[htbp]
		\begin{tikzpicture}[scale=1.5]
\draw[dashed] (0,0) -- (1,0);
\draw[thick] (1/2,1.732/2) -- (1,0);
\draw[thick] (1,0) -- (1/2,-1.732/2);
\draw[thick] (1,0) -- (3/2,1.732/2);
\draw[thick] (2,0) -- (1,0);
\draw[thick] (1,0) -- (3/2,-1.732/2);
\node [below=] at (3/2,-1.732/2+0.7) {$T_2$};
\node [below=] at (3/2,+0.5) {$T_3$};
\node [below=] at (1,0.8) {$T_4$};
\node [below=] at (0.2,0.2) {$T_{5,6}$};
\node [below=] at (0.2,0.2) {$T_{5,6}$};
\node [below=] at (1.,-0.5) {$T_{1}$};
		\end{tikzpicture}
		\qquad
				\begin{tikzpicture}[scale=1.5]
\draw[dashed] (0,0) -- (1,0);
\draw[thick] (1/2,1.732/2) -- (1,0);
\draw[thick] (1,0) -- (1/2,-1.732/2);
\draw[dashed] (1,0) -- (3/2,1.732/2);
\draw[thick] (2,0) -- (1,0);
\draw[thick] (1,0) -- (3/2,-1.732/2);
\node [below=] at (3/2,-1.732/2+0.7) {$T_2$};
\node [below=] at (3/2-0.1,+0.8) {$T_{3,4}$};

\node [below=] at (0.2,0.2) {$T_{5,6}$};
\node [below=] at (0.2,0.2) {$T_{5,6}$};
\node [below=] at (1.,-0.5) {$T_{1}$};
		\end{tikzpicture}
\qquad 
				\begin{tikzpicture}[scale=1.5]
\draw[dashed] (0,0) -- (1,0);
\draw[thick] (1/2,1.732/2) -- (1,0);
\draw[thick] (1,0) -- (1/2,-1.732/2);
\draw[dashed] (1,0) -- (3/2,1.732/2);
\draw[thick] (2,0) -- (1,0);
\draw[dashed] (1,0) -- (3/2,-1.732/2);
\node [below=] at (3/2-0.1,-1.732/2+0.5) {$T_{1,2}$};
\node [below=] at (3/2-0.1,+0.8) {$T_{3,4}$};

\node [below=] at (0.2,0.2) {$T_{5,6}$};

		\end{tikzpicture}
\caption{An illustration of Type II singularities: degenerate vertices (Type II-1, II-2, II-3 from left to right).}
\label{fig:Type II-singularity-type}
	\end{figure}
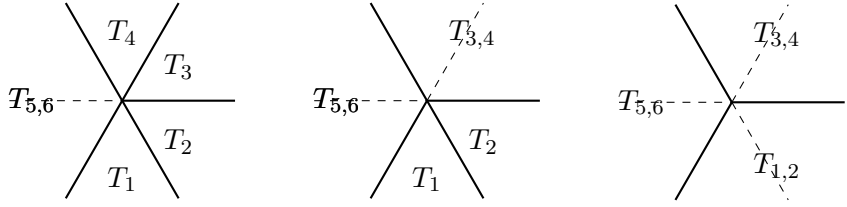
    \end{remark}

Now we consider the sufficiency part. The following decomposition will be frequently used:
$$u = \frac{u\cdot n_s}{n_s\cdot t_{s-1}} t_{s-1} + \frac{u\cdot n_{s-1}}{n_{s-1}\cdot t_{s}} t_s.$$ We first focus on the non-singular vertex, which is illustrated in \Cref{fig:Type II-regular}.

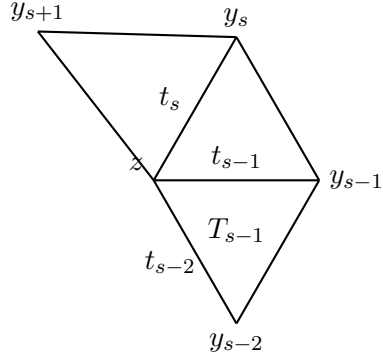
\begin{figure}[htbp]
\begin{tikzpicture}[scale=2.2]
\draw[thick] (0.6/2,1.8/2) -- (1,0);
\draw[thick] (1,0) -- (3/2,1.732/2);
\draw[thick] (2,0) -- (3/2,1.732/2);
\draw[thick] (2,0) -- (1,0);
\draw[thick] (0.6/2,1.8/2)  -- (3/2,1.732/2);
\draw[thick] (1,0) -- (3/2,-1.732/2);
\draw[thick] (2,0) -- (3/2,-1.732/2);
 \node [left=] at (1,0.1) {$z$};
 \node [right=] at (2,0) {$y_{s-1}$};
\node [above=] at (0.6/2,1.8/2)  {$y_{s+1}$};    
\node [above=] at (3/2,1.732/2) {$y_s$};
 \node [below=] at (3/2,-1.732/2) {$y_{s-2}$};
 \node [below=] at (3/2,-1.732/2+0.7) {$T_{s-1}$};
  \node [below=] at (1.1,-1.732/2+0.5) {$t_{s-2}$};
  \node [above=] at (1.5,0) {$t_{s-1}$};
  \node [above=] at (1.1,+1.732/2-0.5) {$t_{s}$};
		\end{tikzpicture}		\caption{An illustration for $T_s$ for a non-singular vertex.}
		\label{fig:Type II-regular}
\end{figure}

\begin{lemma}[Sufficient Conditions for Nonsingular Vertex]
\label{lem:nonsingular-sufficient}
Suppose $z \in \mathcal{V} \setminus (\mathcal{V}_I \cup \mathcal{V}_{II})$. Then, for any $q \in Q_{k-1}^0$, there exists $\sigma \in \Sigma_k^0$ such that:
\begin{enumerate}
\item $\sigma$ is supported in $\st(z)$.
\item $\div \sigma(v) = 0$ for all $v \neq z$.
\item $\div \sigma|_T(z) = q|_T(z)$ for all $T \in \st(z)$.
\end{enumerate}
Moreover, the bound holds:
\[
\|\sigma\|_{H^1} \lesssim \Theta_{II}(z)^{-2} \|q\|_{L^2(\st(z))}.
\]
\end{lemma}

\begin{proof}
The construction proceeds in several steps.

\textbf{Step 1:}  
    Choose an element $T_s \in \st(z)$ such that $|\sin(\theta_s + \theta_{s-1})| \geq \Theta_{II}(z) > 0$ and $|\sin(\theta_s + \theta_{s+1})| \geq \Theta_{II}(z) > 0$. Define
\[
\varphi_{s,-} = \psi_{s-1} \psi_z^2 t_{s-2} t_{s-2}^T, \quad \varphi_{s,+} = \psi_s \psi_z^2 t_{s+1} t_{s+1}^T.
\]
In $T_s$, we have:
\[
\div \varphi_{s,-}|_{T_s}(z) = \frac{\sin(\theta_s + \theta_{s-1})}{h_{s-2}} t_{s-2}, \quad \div \varphi_{s,+}|_{T_s}(z) = \frac{\sin(\theta_s + \theta_{s+1})}{h_s} t_{s+1}.
\]
Since $\sin(\theta_s + \theta_{s\pm 1}) \neq 0$ and $t_{s-2} \neq \pm t_{s+1}$, the vectors $\div \varphi_{s,-}|_{T_s}(z)$ and $\div \varphi_{s,+}|_{T_s}(z)$ form a basis for $\mathbb{R}^2$. Thus, for $q_s(z) = q|_{T_s}(z)$, there exist coefficients $\alpha$ and $\beta$ such that
\[
\alpha \div \varphi_{s,-}|_{T_s}(z) + \beta \div \varphi_{s,+}|_{T_s}(z) = q_s(z),
\]
where
\[
\alpha = \frac{q_s(z) \cdot n_{s+1} h_{s-2}}{(t_{s-2} \cdot n_{s+1}) \sin(\theta_s + \theta_{s-1})}, \quad \beta = \frac{q_s(z) \cdot n_{s-2} h_s}{(t_{s+1} \cdot n_{s-2}) \sin(\theta_s + \theta_{s+1})}.
\]
By shape regularity, $|\alpha|, |\beta| \lesssim \frac{|q_s(z)| h}{\Theta_{II}(z)^2}$. Next, we
\[
\tau_s = \alpha \varphi_{s,-} + \beta \varphi_{s,+},
\]
so
\[
\|\tau_s\|_{H^1} \lesssim \frac{|q_s(z)| h}{\Theta_{II}(z)^2} \lesssim \Theta_{II}(z)^{-2} \|q\|_{L^2(\st(z))}.
\]
Thus, $\tau_s$ satisfies:
\begin{enumerate}
\item $\div \tau_s|_T(z) = 0$ for $T \neq T_s$,
\item $\div \tau_s|_{T_s}(z) = q_s(z)$,
\item $\div \tau_s(y) = 0$ for all vertices $y \neq z$,
\item $\|\tau_s\|_{H^1} \lesssim \Theta_{II}(z)^{-2} \|q\|_{L^2(\st(z))}$.
\end{enumerate}

\textbf{Step 2:}  
Without loss of generality, let $s' = 0$. For each $r$ and $u \in \mathbb{R}^2$, define
\begin{equation}
\label{eq:xi-r}
\xi_r(u) = \psi_r \psi_z^2 \frac{u \cdot n_{r-1}}{n_{r-1} \cdot t_r} \frac{h_r}{\sin \theta_r} t_r t_r^T + \psi_r \psi_z^2 \frac{u \cdot n_r}{n_r \cdot t_{r-1}} \frac{h_{r-1}}{\sin \theta_r} (t_r t_{r-1}^T + t_{r-1} t_r^T).
\end{equation}
By definition, $\xi_r(u)$ is supported in $T_r \cup T_{r+1}$. In $T_r$,
\[
\div \left( \psi_r \psi_z^2 t_r t_r^T \right)|_{T_r}(z) = \frac{\sin \theta_r}{h_r} t_r, \quad \div \left( \psi_{r} \psi_z^2 (t_r t_{r-1}^T + t_{r-1} t_r^T) \right)|_{T_r}(z) = \frac{\sin \theta_r}{h_{r-1}} t_{r-1},
\]
so
\begin{equation} \label{eq:div-xi-r}
\div \xi_r(u)|_{T_r}(z) = u.
\end{equation}
In $T_{r+1}$,
\[
\div \left( \psi_r \psi_z^2 t_r t_r^T \right)|_{T_{r+1}}(z) = \frac{\sin \theta_{r+1}}{h_{r+1}} t_r,
\]
\[
\div \left( \psi_r \psi_z^2 (t_r t_{r-1}^T + t_{r-1} t_r^T) \right)|_{T_{r+1}}(z) = \frac{\sin \theta_{r+1}}{h_{r+1}} t_{r-1} + \frac{\sin(\theta_r + \theta_{r+1})}{h_{r+1}} t_r.
\]
Thus,
\begin{equation} \label{eq:div-xi-r-bound}
\|\xi_r(u)\|_{H^1} \lesssim \frac{|u| h}{\Theta_{II}(z)^2}, \quad |\div \xi_r(u)|_{T_{r+1}}(z)| \lesssim |u|.
\end{equation}
We now define the tensor sequentially:
\[
\tau'_0 = \xi_0(q_0(z)), \quad \tau'_1 = \xi_1(-\div \tau'_0|_{T_1}(z)), \quad \tau'_r = \xi_r(-\div \tau'_{r-1}|_{T_r}(z)), \quad r = 2, \ldots, s-1.
\]
In $T_r$,
\[
\div \tau'_{r-1}|_{T_r}(z) + \div \tau'_r|_{T_r}(z) = 0.
\]
By shape regularity,
\[
|\div \tau'_{r-1}|_{T_r}(z)| \lesssim |\div \tau'_{r-2}|_{T_{r-1}}(z)| \lesssim \cdots \lesssim |q_0(z)|,
\]
yielding
\[
\|\tau'_r\|_{H^1} \lesssim \Theta_{II}(z)^{-2} \|q\|_{L^2(\st(z))}.
\]
Define
\[
\tau_0 = \tau'_0 + \cdots + \tau'_{s-1} + \tau_s(-\div \tau'_{s-1}|_{T_s}(z)),
\]
which satisfies:
\begin{enumerate}
\item $\div \tau_0|_T(z) = 0$ for $T \neq T_0$,
\item $\div \tau_0|_{T_0}(z) = q_0(z),$
\item $\div \tau_0(y) = 0$ for all vertices $y \neq z$,
\item $\|\tau_0\|_{H^1} \lesssim \Theta_{II}(z)^{-2} \|q\|_{L^2(\st(z))}$.
\end{enumerate}

Finally, we can construct $\tau_{s'}$ for all other $T_{s'} \in \st(z)$, $s' \neq s$. Define $\sigma = \sum_{s'} \tau_{s'}$, which completes the proof.

\end{proof}

In what follows, we consider the case of the two types of singularities, which are critical to the overall analysis. We proceed by examining each type of singularity separately. The following lemma focuses on Type I singularities, indicating that \eqref{eq:lag-cond-typeI} is a sufficient condition. 

\begin{lemma}[Sufficient Conditions for Type I Singularity]
\label{lem:typeI-sufficient}
Suppose $z \in \mathcal{V}_I$ and $q \in Q_{k-1}^0$ satisfies \eqref{eq:lag-cond-typeI}. Then, there exists $\sigma \in \Sigma_k^0$ such that:
\begin{enumerate}
\item $\sigma$ is supported in $\st(z)$.
\item $\div \sigma(v) = 0$ for all $v \neq z$.
\item $\div \sigma|_T(z) = q|_T(z)$ for all $T \in \st(z)$.
\end{enumerate}
Moreover, the bound holds:
\[
\|\sigma\|_{H^1} \lesssim \|q\|_{L^2(\st(z))}.
\]
\end{lemma}

\begin{proof}
From Step 2 of \Cref{lem:nonsingular-sufficient}, recall the construction of $\xi_r(u)$ defined in \eqref{eq:xi-r}. For $z \in \mathcal{V}_I$, the geometric configuration implies four elements $\sin \theta_{r+2} = \sin \theta_r$ and $\sin(\theta_r + \theta_{r+1}) = 0$. Thus, it holds that
\[
\div \left( \psi_r \psi_z^2 (t_r t_{r-1}^T + t_{r-1} t_r^T) \right)|_{T_{r+1}}(z) = \frac{\sin \theta_{r+1}}{h_{r+1}} t_{r-1}.
\]
Since $\div \left( \psi_r \psi_z^2 t_r t_r^T \right)|_{T_{r+1}}(z) = \frac{\sin \theta_{r+1}}{h_{r+1}} t_r$, the divergence of $\xi_r(u)$ in $T_{r+1}$ becomes:
\begin{equation}
\label{eq:div-xi-r}
\div \xi_r(u)|_{T_{r+1}}(z) = u.
\end{equation}
It then follows from \eqref{eq:xi-r} that $\div \xi_r(u)|_{T_r}(z) = u$. Now we define:
\[
\tau_1 := \xi_1(q_1(z)), \quad \tau_2 := \xi_2(q_2(z) - q_1(z)), \quad \tau_3 := \xi_3(q_3(z) - q_2(z) + q_1(z)).
\]
Since \eqref{eq:lag-cond-typeI} gives $q_1(z) + q_3(z) = q_2(z) + q_4(z)$, it follows that $q_3(z) - q_2(z) + q_1(z) = q_4(z)$. Thus, $\tau_3 = \xi_3(q_4(z))$. Set $\sigma = \tau_1 + \tau_2 + \tau_3$. Then:
\begin{itemize}
\item In $T_1$: $\div \sigma|_{T_1}(z) = \div \tau_1|_{T_1}(z) = q_1(z)$, since $\tau_2, \tau_3$ are supported in $T_2 \cup T_3$ and $T_3 \cup T_4$, respectively.
\item In $T_2$: $\div \sigma|_{T_2}(z) = \div \tau_1|_{T_2}(z) + \div \tau_2|_{T_2}(z) = q_1(z) + (q_2(z) - q_1(z)) = q_2(z)$.
\item In $T_3$: $\div \sigma|_{T_3}(z) = \div \tau_2|_{T_3}(z) + \div \tau_3|_{T_3}(z) = (q_2(z) - q_1(z)) + q_4(z) = q_3(z)$, using \eqref{eq:lag-cond-typeI}.
\item In $T_4$: $\div \sigma|_{T_4}(z) = \div \tau_3|_{T_4}(z) = q_4(z)$.
\end{itemize}
Since each $\xi_r$ is supported in $T_r \cup T_{r+1}$ and $\div \xi_r(y) = 0$ for other vertices $y \neq z$, $\sigma$ satisfies conditions (1) and (2). Finally, note that $\|\xi_r(u)\|_{H^1} \lesssim |u| h$ (as $\sin \theta_r$ is bounded away from zero). Thus:
\[
\|\tau_r\|_{H^1} \lesssim \|q\|_{L^2(\st(z))}, \quad \|\sigma\|_{H^1} \leq \sum_{r=1}^3 \|\tau_r\|_{H^1} \lesssim \|q\|_{L^2(\st(z))},
\]
completing the proof.
\end{proof}
For Type II singularities, we adopt a similar approach to that used for Type I singularities. However, the analysis of Type II singularities introduces additional complexity due to their distinct geometric configurations, which are classified into nondegenerate and degenerate cases. Each configuration requires careful consideration of the local mesh geometry. The following lemma addresses the nondegenerate Type II singularity case.

\begin{lemma}[Sufficient Conditions for Nondegenerate Type II Singularity]
\label{lem:suff-typeII}
Suppose $z \in \mathcal{V}_{II}$ is a nondegenerate vertex and $q \in Q_{k-1}^0$ satisfies \eqref{eq:lag-cond-typeII}. Then, there exists $\sigma \in \Sigma_k^0$ such that:
\begin{enumerate}
\item $\sigma$ is supported in $\st(z)$.
\item $\div \sigma|_T(z) = q|_T(z)$ for all $T \in \st(z)$.
\item $\div \sigma|_T(y) = 0$ for all $y \neq z$, $T \in \st(z)$.
\end{enumerate}
Moreover, the bound holds:
\[
\|\sigma\|_{H^1} \lesssim \|q\|_{L^2(\st(z))}.
\]
\end{lemma}

\begin{proof}
The construction is decomposed into two steps. First, we prove that there exists $\sigma \in \Sigma_k^0$ such that in each $T_i$, it holds that 
\begin{equation}
    (\div \sigma|_{T_i}(z) - q_i(z)) \cdot n_{i+1} = 0.
\end{equation}

Consider $\eta_s = \psi_s\psi_z^2 t_s t_s^T$. Note that $\eta_s$ is supported in $T_s \cup T_{s+1}$. In $T_s$, we have
$$ \div \eta_s|_{T_s} \cdot n_{s+1} = \frac{\sin \theta_s \sin \theta_{s+1}}{h_s} = \frac{\sin \theta_{s+1}}{|\overline{zy_s}|}, $$
while in $T_{s+1}$, 
$$ \div \eta_s|_{T_{s+1}} \cdot n_{s+2} = \frac{\sin \theta_s}{|\overline{zy_s}|}, $$
Therefore, by setting $\displaystyle \tau_s = \frac{|\overline{zy_s}|}{\sin \theta_s \sin \theta_{s+1}} \eta_s$, we have 

$$\div \tau_s|_{T_s} \cdot n_{s+1} \sin \theta_s = \div \tau_s|_{T_{s+1}} \cdot n_{s+2} \sin \theta_{s+1} = 1.$$
This is illustrated in \Cref{fig:first-step-nondegenerate}. Here the number inside each $T_s$ records the value $\div \tau|_{T_{s}} \cdot n_{s+1} \sin \theta_{s}$.

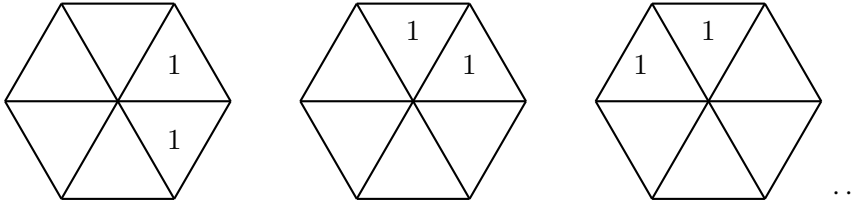
\begin{figure}[htbp]
				\begin{tikzpicture}[scale=1.5]
\draw[thick] (0,0) -- (1,0);
\draw[thick] (1/2,1.732/2) -- (1,0);
\draw[thick] (0,0) -- (1/2,1.732/2);
\draw[thick] (0,0) -- (1/2,-1.732/2);
\draw[thick] (1,0) -- (1/2,-1.732/2);
\draw[thick] (1,0) -- (3/2,1.732/2);
\draw[thick] (2,0) -- (3/2,1.732/2);
\draw[thick] (2,0) -- (1,0);
\draw[thick] (1/2,1.732/2) -- (3/2,1.732/2);
\draw[thick] (1/2,-1.732/2) -- (3/2,-1.732/2);
\draw[thick] (1,0) -- (3/2,-1.732/2);
\draw[thick] (2,0) -- (3/2,-1.732/2);
\node [below=] at (3/2,-1.732/2+0.7) {$1$};
\node [below=] at (3/2,+0.5) {$1$};

		\end{tikzpicture}
		\qquad 
				\begin{tikzpicture}[scale=1.5]
\draw[thick] (0,0) -- (1,0);
\draw[thick] (1/2,1.732/2) -- (1,0);
\draw[thick] (0,0) -- (1/2,1.732/2);
\draw[thick] (0,0) -- (1/2,-1.732/2);
\draw[thick] (1,0) -- (1/2,-1.732/2);
\draw[thick] (1,0) -- (3/2,1.732/2);
\draw[thick] (2,0) -- (3/2,1.732/2);
\draw[thick] (2,0) -- (1,0);
\draw[thick] (1/2,1.732/2) -- (3/2,1.732/2);
\draw[thick] (1/2,-1.732/2) -- (3/2,-1.732/2);
\draw[thick] (1,0) -- (3/2,-1.732/2);
\draw[thick] (2,0) -- (3/2,-1.732/2);

\node [below=] at (3/2,+0.5) {$1$};
\node [below=] at (1,0.8) {$1$};

		\end{tikzpicture}
\qquad 
				\begin{tikzpicture}[scale=1.5]
\draw[thick] (0,0) -- (1,0);
\draw[thick] (1/2,1.732/2) -- (1,0);
\draw[thick] (0,0) -- (1/2,1.732/2);
\draw[thick] (0,0) -- (1/2,-1.732/2);
\draw[thick] (1,0) -- (1/2,-1.732/2);
\draw[thick] (1,0) -- (3/2,1.732/2);
\draw[thick] (2,0) -- (3/2,1.732/2);
\draw[thick] (2,0) -- (1,0);
\draw[thick] (1/2,1.732/2) -- (3/2,1.732/2);
\draw[thick] (1/2,-1.732/2) -- (3/2,-1.732/2);
\draw[thick] (1,0) -- (3/2,-1.732/2);
\draw[thick] (2,0) -- (3/2,-1.732/2);
\node [below=] at (1,0.8) {$1$};
\node [below=] at (0.4,0.5) {$1$};
		\end{tikzpicture}
        $\cdots$
        
\caption{An illustration supporting the proof in the first step.}
\label{fig:first-step-nondegenerate}
	\end{figure}


Define 
$$\sigma_1 = (q_1(z) \cdot n_{2}) \tau_1,$$
$$\sigma_r = (q_r(z) \cdot n_{r+1} - \div \sigma_{r-1}|_{T_{r}} \cdot n_{r+2}) \tau_{r}, \text{ for } r = 2,3,4,5.$$

Consider $\sigma = \sigma_1+\cdots+ \sigma_5.$ Then it holds that for $i = 1,2,3,4,5$,
 $$(\div \sigma|_{T_i}(z) - q_i(z)) \cdot n_{i+1} = 0.$$
 Thanks for \eqref{eq:lag-cond-typeII}, the above equation also holds for $i = 6$. Thus, we finish the first step of construction. 

For the second step, we construct $\tilde{\sigma}$ such that 
 $(\div \sigma|_{T_i}(z) + \div \tilde\sigma|_{T_i}(z) -  q_i(z))  = 0.$
Thanks to the first step, we already have 
$\div \sigma|_{T_i}(z) - q_i(z) = \alpha_i t_{i+1}$ for some constant $\alpha_i$. 
Now consider $\tilde \sigma_s = \psi_s \psi_z^2 t_{s+1} t_{s+1}^T$, which is supported in $T_s \cup T_{s+1}$. Then, direct computation yields  $\div \tilde \sigma_s|_{T_{s+1}} = 0$, while $\div \tilde \sigma_s|_{T_{s}} = \frac{\sin \theta_{s-1}}{h_s} t_{s+1}$. Therefore, 
taking $\tilde \sigma = \sum_{i = 1}^6 \frac{\alpha_s h_s}{\sin \theta_{s-1}}\tilde \sigma_s $ we finish the proof for the nondegenerate case.
\end{proof}

The degenerate cases are similar, and the complete proof are displayed in \Cref{sec:appendix}. 

Combining all the above results, we have the following theorem.
\begin{theorem}
\label{thm:lagrange}
It holds that 
    \begin{equation}\begin{split} 
    \div \Sigma_k^0 = \Big\{   &q \in Q_{k-1}^{-1}  :\\ &  J_z : \nabla q_{1}(v)+  J_z : \nabla  q_{3}(v) =  J_z : \nabla  q_{2}(v)+  J_z : \nabla  q_{4}(v),\,\, \forall v \in \mathcal V_{I} \\ &   q_{1}(v)+   q_{3}(v) =  q_{2}(v)+  q_{4}(v),\,\, \forall v \in \mathcal V_{I}  \\
   &  \sum_{i = 1}^6 (-1)^i \sin\theta_i q_i(v) \cdot n_{i+1} = 0, \,\, \forall v \in \mathcal V_{II}.\Big\}
    \end{split}\end{equation}
    Moreover, if for all $x \in \mathcal V \setminus \mathcal V_{I} \setminus \mathcal V_{II}$, $\Theta_{II}(z)$ has a uniform lower bound $\theta_{II} > 0$, and $\mathcal T$ is shape regular. Then for each $ q \in \div \Sigma_k^0$, there exists $ \sigma \in \Sigma_k^0$ such that $\div \sigma = q$ and $\|\sigma\|_{H^1} \lesssim \theta_{II}^{-2} \|q\|_{L^2}$.
\end{theorem}
\begin{proof} 
The proof is similar to \Cref{thm:hermite}.
\end{proof}	

\subsection{Error Estimates}
\label{sec:error}

Based on the inf-sup stability conditions, the following error estimates are standard. Note that the constant for inf-sup condition and coercivity are independent of both parameters $\lambda$ and $\iota$. Therefore, the error estimates are also parameter robust \eqref{lsg-discrete}. 
\begin{theorem}
    Let $(\sigma,u)\in \Sigma\times Q$ be the exact solution of problem \eqref{lsg-weak}. 
    Let $\Sigma_h \times Q_h$ be one of the following pairs:

    \begin{enumerate}
        \item $\Sigma_k^2 \times Q_k^1$, provided $k \ge 7$;
        \item $\Sigma_k^1 \times Q_k^0$, provided $k \ge 7$ and $\Theta_I(z) > \theta_I$ for all interior vertices and $\theta_I > 0$;
        \item $\Sigma_k^0 \times Q_k^{-1}$, provided $k \ge 7$ and $\Theta_{I}(z) > \theta_{I}$ and $\Theta_{II}(z) > \theta_{II}$ for all interior vertices and $\theta_I, \theta_{II} > 0$.
    \end{enumerate}
    Let $(\sigma_h,u_h)\in \Sigma_h \times Q_h$ be the finite element solution of \eqref{lsg-discrete}.
    Then 
    \begin{equation}
    \|\sigma-\sigma_h\|_{\iota}+\|u-u_h\|_{L^2}\lesssim \inf_{\tau_h\in \Sigma_{h},~v_h\in Q_{h}}(\|\sigma-\tau_h\|_{\iota}+\|u-u_h\|_{L^2}).
    \end{equation}
Here the hidden constant is independent with $\iota$ and $\lambda$.
\end{theorem}

As a corollary, suppose that $\sigma\in H^{k+1}(\Omega;\mathbb{S}^2)$ and $u\in H^k(\Omega;\mathbb R^2)$, then we have
        \begin{equation}
    \|\sigma-\sigma_h\|_{\iota}+\|u-u_h\|_{L^2}\lesssim h^k(\|\sigma\|_{H^{k+1}}+\|u\|_{H^{k}}).
    \end{equation}
{Note that the norm of $\sigma$ and $u$ might depend on $\iota$.}

\subsection{The smoothed stress complex with spline spaces} Denote by the superspline space $U_{k+2}^3$ and $U_{k+2}^2$ as 
$$U_{k+2}^r: =\{ u \in C^2(\Omega): u|_T \in P_{k+2}, u \in C^r(\mathcal V)\}.$$
Therefore, the following space is closed. 
\begin{equation*}
\label{eq:complex-spline}
0 ~ {\hookrightarrow}~ U^{r+2}_{k+2}(\mathcal T) \xrightarrow{\curl\curl^T} \Sigma_{k}^r (\mathcal T) \xrightarrow{\div} Q_{k-1}^{r-1}(\mathcal T) \to 0.
\end{equation*}
When $r = 2$, the space $U_{k+2}^4$ admits a finite element construction. It can be readily checked with the help of bubble complexes, that the cohomology of \eqref{eq:complex-spline} is isomorphic to $\mathbf{P}_1 \otimes \mathcal H_{dR}(\Omega)$, see \cite{ChristiansenHu2023, HuLiangLin23, HuLinZhang25}. However, when $r = 0$ or $r = 1$, the space $U_{k+2}^2$ (or $U_{k+2}^3$) does not admit a finite element construction \cite{HuLinWuYuan25}, which leads to the difficulty identifying the cohomology. The exception is when $\Omega$ is homeomorphic to a disk. In that case, an analysis about the scalar potentials can make sure that the following sequence is exact.
\begin{equation}\label{eq:spline-complex}
\mathbf{P}_1 ~ {\hookrightarrow}~ U^{r+2}_{k+2}(\mathcal T) \xrightarrow{\curl\curl^T} \Sigma_{k}^r (\mathcal T) \xrightarrow{\div} \div \Sigma_{k}^{r}(\mathcal T) \to 0.
\end{equation}
Similar exactness argument can be established for the Scott-Vogelius elements with applications, cf. \cite{GuzmanScott18}.

\begin{theorem}
    When $\Omega$ is topologically trivial, then \eqref{eq:spline-complex} is exact for $r= 0, k \ge 5$, $r = 1, k \ge 7$ and $r = 2, k \ge 7$.
\end{theorem}

\begin{proof}
It can be checked that if $\sigma  \in \Sigma_k^r(\mathcal T)$ and $\div \sigma = 0$, then there exists $\psi \in U_{k+2}^{r+2}(\mathcal T)$ such that $\curl \curl^T \psi = \sigma.$ 
{
It suffices to check the dimension count. By Alfeld-Schumaker's formula \cite{Alfeld1987}, it holds that for $k \ge 5$, we have 
\begin{equation}
\begin{split}
\dim U_{k+2}^2(\mathcal T) = &     \binom{k+4}{2} + \binom{k+1}{2} |\mathcal E^{\circ}| - \big[ \binom{k+4}{2} - 6\big] |\mathcal V^{\circ}| + 3|\mathcal V_{I}| + |\mathcal V_{II}| \\
 = &  6|\mathcal V| + (3k-6)|\mathcal E| +\binom{k+1}{2} |\mathcal F| + 3|\mathcal V_{I}| + |\mathcal V_{II}|.
\end{split}\end{equation}
Here we use $|\mathcal V| - |\mathcal E| + |\mathcal F| = 1$, $|\mathcal V^{\partial}| = |\mathcal E^{\partial}|$, and $3|\mathcal F| = |\mathcal E| + |\mathcal E^{\circ}|$.
Therefore, we have 
\begin{equation}
\begin{split}&  \dim U_{k+2}^2(\mathcal T) - \dim \Sigma_k^0(\mathcal T) + \dim \div \Sigma_{k}^{0}(\mathcal T) \\
 = & \dim U_{k+2}^2(\mathcal T) - \dim \Sigma_k^0(\mathcal T) + \dim Q_{k-1}^{-1}(\mathcal T)  + 3|\mathcal V_{I}| + |\mathcal V_{II}| \\ 
= & 3 (|\mathcal V| - |\mathcal E| + |\mathcal F|) = 3.
\end{split}
\end{equation}

There are also some results about the dimension formula in the superspline space $U_{k+2}^3(\mathcal T)$. In \cite{toshniwal2023} we have 
\begin{equation}
\begin{split} \dim U_{k+2}^3 = &\binom{k+4}{2} + [(k-1)^2 - \binom{k-2}{2}] |\mathcal E^{\circ}| - [ \binom{k+4}{2} - 10] |\mathcal V^{\circ}| + |\mathcal V_{I}|\\
= & 10 |\mathcal V| + (3k-12) |\mathcal E| + [\binom{k-5}{2} - 3] |\mathcal F| + |\mathcal V_{I}|
\end{split}
\end{equation} 
for $k$ sufficiently large. Therefore, we have:
\begin{equation}\dim  U_{k+2}^3(\mathcal T) - \dim \Sigma_k^1(\mathcal T) + \dim \div \Sigma_k^1(\mathcal T) = 3.
\end{equation}
On the other side, the above calculation shows that the dimension formula of $\dim U_{k+2}^3(\mathcal T)$ is valid when $k \ge 7$.

}
\end{proof}

The general cohomology and machinery of the complexes with splines are beyond the scope of this paper.

\section{Numerical Experiments}
\label{sec:numerics}

In this section, we provide some numerical examples for verification and further investigation. We start from an exact solution of $(\sigma_0, u_0)$, corresponds to the case when $\iota = 0$. The right-hand side will be constructed by $f = \div \sigma_0$.

For practical concerns, we pay more attention on the Hermite case $\Sigma_k^1\times Q_{k-1}^0$ and Lagrange case $\Sigma_k^0\times Q_{k-1}^{-1}$ for $k = 4$, whose polynomial degree is lower than assumed in the previous theorems.

\subsection{Verification}

We take an exact solution $u_0$ for problem \eqref{eq:iota0} with sufficiently boundary vanishing conditions. Specifically, we choose $u_{0} = \curl(\sin^3\pi x\sin^3\pi y)+(1/\lambda)x^3(1-x)^3y^3(1-y)^3.$ 
Note that $u_0\in H_0^2([0,1]^2;\mathbb{R}^2)$, therefore for any divergence-free matrix-valued functions $\tau$, we have
$$
(\nabla\mathcal{A}\sigma(u_0),\nabla\tau)=(\nabla\mathrm{symgrad}(u_0),\nabla\tau)=(u_0,\div   \Delta \tau)=0.
$$
By \eqref{equation:bdlayer-diff}, we have $\sigma_{\iota}=\sigma_0:=\mathcal{A}^{-1}\mathrm{symgrad}(u_0).$ 
However, the solution of $u_{\iota}$ do not have an analytic expression. Therefore, we use computational results in $\Sigma_5^2 \times Q_{4}^1$ and \texttt{hsize}$=1/128$ as reference solutions. Especially when $\iota = 0$, we compute the error by the exact solution (available only when $\iota = 0$). We fix $\lambda = 10^5$, $\mu = 0.3$.

To resolve the singularities caused by geometry, we use some specific refinements in mesh generation. 
For the pair $\Sigma_k^1\times Q_{k-1}^0$, we choose the Hsieh-Clough-Tocher split; while for $\Sigma_k^0\times Q_{k-1}^{-1}$,we choose Morgan-Scott split mesh, respectively, cf. \Cref{fig:mesh}. Hereafter, \texttt{hsize} refers to the mesh size before refinement. As a consequence, the actual size of the mesh would be smaller than \texttt{hsize}.



\begin{figure}[htbp]
    \centering
    \includegraphics[width=0.4\linewidth]{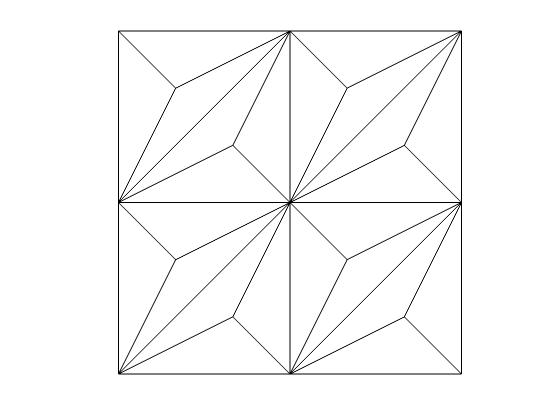}
    \includegraphics[width=0.4\linewidth]{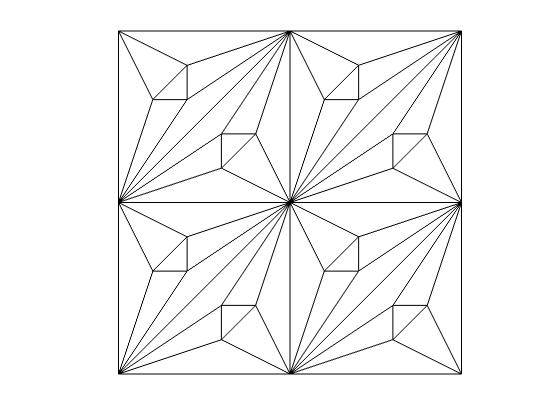}
    \caption{Meshes refined by Hsieh-Clough-Tocher split and Morgan-Scott split}
    \label{fig:mesh}
\end{figure}

The error between the reference solution and the numerical solutions is denoted as $\|E_{\sigma}\|_{\iota}$ and $\|E_{u}\|_{L^2}$. 
The numerical results of $\Sigma_4^0\times Q_3^{-1}$ and $\Sigma_4^1\times Q_3^{0}$ is listed in \Cref{tab:0,tab:1}.

\begin{center}
\begin{table}
        \caption{Numerical experiments of $\Sigma_4^1\times Q_3^{0}$ on Hsich-Clough-Tocher Split }
    \begin{tabular}{cccccccc}
    \hline
         hsize & 1/2 & 1/4&1/6&1/8&1/10&1/12&1/14 \\
         \hline
    &&&$\iota=1$&&&&\\
    \hline
   $\|E_{\sigma}\|_{\iota}$    & 2.76e+01 & 3.68e+00 &  1.18e+00 & 4.85e-01 & 2.30e-01 &   1.21e-01 &   6.93e-02 \\
     order  &    --    &     2.91 &      2.80 &     3.09 &     3.34 &       3.51 &       3.63  \\
    $\|E_{u}\|_{L^2}$   & 9.99e+01 & 1.03e+01 & 4.45e+00 & 2.47e+00 & 1.43e+00 & 8.58e-01 & 5.34e-01  \\
     order   &    --    & 3.27     & 2.08     & 2.04     & 2.44     & 2.81     & 3.08       \\
     \hline
     &&&$\iota=0.01$&&&&\\
     \hline
   $\|E_{\sigma}\|_{\iota}$  & 5.26e+00 & 5.97e-01 &  1.32e-01 & 4.22e-02 & 1.70e-02 &   8.05e-03 &   4.27e-03 \\
    order&    --    &     3.14 &      3.72 &     3.96 &     4.07 &       4.11 &       4.12 \\
     $\|E_{u}\|_{L^2}$   & 1.09e+00 & 1.09e-01 & 4.59e-02 & 2.52e-02 & 1.45e-02 & 8.64e-03 & 5.36e-03\\
    order &    --    & 3.32     & 2.14     & 2.09     & 2.48     & 2.83     & 3.10   \\
    \hline
       &&&$\iota=1.0e-4$&&&& \\
     \hline
     $\|E_{\sigma}\|_{\iota}$   & 5.14e+00 & 5.89e-01 &  1.29e-01 & 4.09e-02 & 1.64e-02 &   7.78e-03 &   4.13e-03\\
     order  &    --     &     3.13 &      3.74 &     4.00 &     4.08 &       4.11 &       4.11    \\
       $\|E_{u}\|_{L^2}$     & 1.24e-01 & 1.24e-02 & 2.82e-03 & 9.49e-04 & 4.02e-04 & 1.97e-04 & 1.08e-04\\
       order&    --    & 3.32     & 3.65     & 3.78     & 3.85     & 3.90     & 3.92     \\
    \hline
           &&&$\iota=1.0e-6$&&&& \\
     \hline
     $\|E_{\sigma}\|_{\iota}$   &5.13e+00 & 5.87e-01 &  1.28e-01 & 4.05e-02 & 1.62e-02 &   7.68e-03 &   4.08e-03\\
     order    &    --    &     3.13 &      3.75 &     4.01 &     4.09 &       4.11 &       4.11  \\
       $\|E_{u}\|_{L^2}$   & 1.16e-01 & 1.19e-02 &  2.53e-03 & 7.93e-04 & 3.17e-04 &   1.50e-04 &   8.03e-05  \\
       order&    --     &     3.29 &      3.81 &     4.04 &     4.10 &       4.09 &       4.07   \\
    \hline
           &&&$\iota=1.0e-8$&&&& \\
     \hline
     $\|E_{\sigma}\|_{\iota}$    & 5.13e+00 & 5.87e-01 &  1.28e-01 & 4.05e-02 & 1.62e-02 &   7.68e-03 &   4.08e-03 \\
     order   &    --     &     3.13 &      3.75 &     4.01 &     4.09 &       4.11 &       4.11     \\
       $\|E_{u}\|_{L^2}$    & 1.16e-01 & 1.18e-02 &  2.53e-03 & 7.91e-04 & 3.17e-04 &   1.50e-04 &   8.02e-05\\
       order  &    --    & 3.29     & 3.81     & 4.04     & 4.10     & 4.09     & 4.07     \\
    \hline
           &&&$\iota=0$&&&& \\
     \hline
     $\|E_{\sigma}\|_{\iota}$    & 5.13e+00 & 5.87e-01 &  1.28e-01 & 4.05e-02 & 1.62e-02 &   7.68e-03 &   4.08e-03  \\
     order   &    --    &3.13 &      3.75 &     4.01 &     4.10 &       4.11 &       4.11 \\
       $\|E_{u}\|_{L^2}$      & 1.16e-01 & 1.18e-02 &  2.53e-03 & 7.91e-04 & 3.17e-04 &   1.50e-04 &   8.02e-05 \\
       order &     --    &   3.29 &      3.81 &     4.04 &     4.10 &       4.09 &       4.07 \\
    \hline
    \end{tabular}
        \label{tab:0}
    \end{table}
\end{center}

\begin{center}
\begin{table}
 \caption{Numerical experiments of $\Sigma_4^0\times Q_3^{-1}$ on Morgan-Scott Split }
    \begin{tabular}{cccccccc}\hline
      hsize & 1/2 & 1/4&1/6&1/8&1/10&1/12&1/14 \\
      \hline
    &&&$\iota=1$&&&&\\
    \hline
   $\|E_{\sigma}\|_{\iota}$   & 1.92e+01 & 2.40e+00 &  6.25e-01 & 2.31e-01 & 1.04e-01 &   5.39e-02 &   3.10e-02\\
     order&      --  &     3.00 &      3.32 &     3.46 &     3.56 &       3.63 &       3.58  \\
    $\|E_{u}\|_{L^2}$   & 7.79e+01 & 1.27e+01 &  4.01e+00 & 1.63e+00 & 7.83e-01 &   4.20e-01 &   2.45e-01  \\
     order &      -- & 2.62 & 2.83 & 3.13 & 3.29 & 3.41 & 3.51    \\
     \hline
     &&&$\iota=0.01$&&&&\\
     \hline
   $\|E_{\sigma}\|_{\iota}$& 3.51e+00 & 3.37e-01 &  7.00e-02 & 2.23e-02 & 9.13e-03 &   4.39e-03 &   2.36e-03\\
    order& -- &     3.38 &      3.88 &     3.97 &     4.00 &       4.02 &       4.03   \\
     $\|E_{u}\|_{L^2}$    & 8.19e-01 & 1.28e-01 &  4.04e-02 & 1.64e-02 & 7.85e-03 &   4.21e-03 &   2.45e-03  \\
    order   & -- &     2.67 &      2.85 &     3.14 &     3.30 &       3.42 &       3.51    \\
    \hline
       &&&$\iota=1.0e-04$&&&& \\
     \hline
     $\|E_{\sigma}\|_{\iota}$     & 3.40e+00 & 3.24e-01 &  6.76e-02 & 2.17e-02 & 8.92e-03 &   4.30e-03 &   2.32e-03   \\
     order   &    --    &     3.39 &      3.86 &     3.95 &     3.98 &       4.00 &       4.01      \\
       $\|E_{u}\|_{L^2}$   & 7.43e-02 & 6.87e-03 &  1.46e-03 & 4.81e-04 & 2.02e-04 &   9.92e-05 &   5.44e-05 \\
       order &      --   & 3.44 & 3.81 & 3.87 & 3.89 & 3.90 & 3.91     \\
    \hline
           &&&$\iota=1.0e-06$&&&& \\
     \hline
     $\|E_{\sigma}\|_{\iota}$    & 3.39e+00 & 3.21e-01 &  6.68e-02 & 2.14e-02 & 8.80e-03 &   4.24e-03 &   2.29e-03 \\
     order   &    --             &     3.40 &      3.87 &     3.95 &     3.98 &       4.00 &       4.01    \\
       $\|E_{u}\|_{L^2}$    & 6.88e-02 & 6.39e-03 & 1.34e-03 & 4.32e-04 & 1.79e-04 & 8.67e-05 & 4.70e-05 \\
       order &      --     & 3.43     & 3.86     & 3.93     & 3.95     & 3.97     & 3.98     \\
    \hline
           &&&$\iota=1.0e-08$&&&& \\
     \hline
     $\|E_{\sigma}\|_{\iota}$   & 3.39e+00 & 3.21e-01 &  6.68e-02 & 2.14e-02 & 8.80e-03 &   4.24e-03 &   2.29e-03   \\
     order  &    --    &     3.40 &      3.87 &     3.95 &     3.98 &       4.00 &       4.01   \\
       $\|E_{u}\|_{L^2}$    & 6.87e-02 & 6.39e-03 & 1.34e-03 & 4.32e-04 & 1.79e-04 & 8.67e-05 & 4.70e-05\\
       order   &    --    &   3.43   &   3.86   &   3.93   &   3.95   &   3.97   &   3.98   \\
    \hline
           &&&$\iota=0$&&&& \\
     \hline
     $\|E_{\sigma}\|_{\iota}$     & 3.39e+00 & 3.21e-01 &  6.68e-02 & 2.14e-02 & 8.80e-03 &   4.24e-03 &   2.29e-03 \\
     order  & --- &     3.40 &      3.87 &     3.95 &     3.98 &       4.00 &       4.01    \\
       $\|E_{u}\|_{L^2}$    & 6.87e-02 & 6.39e-03 &  1.34e-03 & 4.32e-04 & 1.79e-04 &   8.67e-05 &   4.70e-05 \\
       order & --- & 3.43 & 3.86 & 3.93 & 3.95 & 3.93 &  3.98  \\
    \hline
    \end{tabular}
         \label{tab:1}
    \end{table}
\end{center}

\subsection{Effect of $\partial_n \sigma_0$}

As suggested by \Cref{prop:bdlayer-type}, the convergence mode will be different for the case when $\partial_n \sigma =0$ and $\partial_n \sigma \neq 0$. This is further investigated in this subsection, in particular,  $H^2$ stability of $\sigma_h$, where we choose Lam\'e constants $\mu=0.3,\lambda=1$ be mild.

To achieve this, we set
$$
\tilde u_0=(\sin x\sin y(\sin x-\sin 1 )(\sin y-\sin 1) ,\sin x\sin y(e^x-e )(e^y-e) )^T.
$$
And $\tilde \sigma_0=\mathcal{A}^{-1}\mathrm{symgrad} (\tilde u_0).$ Note that for this pair, it holds that $\partial_n\tilde \sigma_0|_{\partial \Omega}\neq 0.$  In contrast, the pair $(\sigma_0, u_0)$ discussed in the previous subsection have the consistent boundary condition, namely, $\partial_n \sigma_0 = 0$.  We test the two right hand sides on $1/2,1/4,\cdots,1/128$ size mesh with the pair $(\Sigma^2_5,Q^1_4),$ and $\iota$ set as $10^{-4},10^{-5} $ and $10^{-6}$.  Figure \ref{fig:sigma0} and \ref{fig:tildesigma0} show the maximum value of $\sigma_{h,xy}$ on each mesh nodes for two different boundary conditions ($ \partial_n \sigma_0 = 0$ and $\partial \tilde \sigma_0 \neq 0$), respectively. For $(\sigma_0, u_0)$, the computational second-order derivatives in \Cref{fig:sigma0} are stable with respect to $\iota$ and $h$, while the discrete $W^{2,\infty}$ stability is not observed in \Cref{fig:tildesigma0}.

Current numerical experiments indicate that the optimal convergence rate can be achieved even when $\iota \to 0$, provided the left-hand side is smooth enough. A thorough analysis of the continuous problem and more theoretical explanations for the numerical results are beyond the scope of this paper and will be addressed in future work.

 \begin{figure}[htbp]
\begin{minipage}{0.4\textwidth}
\centering
\begin{tikzpicture}[scale = 0.5]
\begin{axis}[
    width=12cm,
    xlabel={$-\log_2~\texttt{hsize}$},
    ylabel={$\|D^2\sigma_{\iota,h}\|_{\infty,h}$},
    legend pos=north east,
    grid=both,
    minor tick num=5,
    title={$\lambda=1,\mu=0.3$ },
    legend cell align={left}
]
\addplot+[mark=triangle*, thick] coordinates {
(1, 794.803)
(2, 546.8935)
(3, 526.3693)
(4, 525.9476)
(5, 526.0028)
(6, 526.0086)
};
\addlegendentry{$\|D^2\sigma_{\iota,h}\|_{\infty,h},~\iota=1.0e-03$}

\addplot+[mark=o, thick] coordinates {
(1, 881.8342)
(2, 561.9746)
(3, 526.2296)
(4, 525.9229)
(5, 526.0019)
(6, 526.0086)
};
\addlegendentry{$\|D^2\sigma_{\iota,h}\|_{\infty,h},~\iota=1.0e-04$ }

\addplot+[mark=triangle*, thick] coordinates {
(1, 946.0311)
(2, 597.6946)
(3, 526.1217)
(4, 525.8835)
(5, 525.9981)
(6, 526.0085)
};
\addlegendentry{$\|D^2\sigma_{\iota,h}\|_{\infty,h},~\iota=1.0e-05$ }

\addplot+[mark=square*, thick] coordinates {
(1, 953.6676)
(2, 603.4002)
(3, 526.0994)
(4, 525.8724)
(5, 525.9952)
(6, 526.0082)
};
\addlegendentry{$\|D^2\sigma_{\iota,h}\|_{\infty,h},~\iota=1.0e-06$}

\end{axis}
\end{tikzpicture}
\end{minipage}
\begin{minipage}{0.4\textwidth}
\centering
 \includegraphics[width=1.2\linewidth]{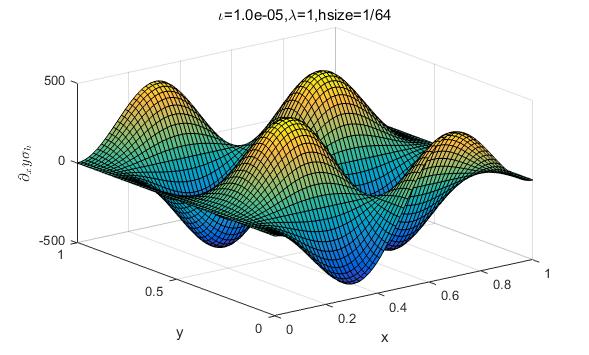}
\end{minipage}

 \caption{ The error of  $\|D^2 \sigma_{\iota,h} \|_{\infty,h}$ (left) and a profile of $\partial_{xy}  \sigma_{\iota,h}$ (right) for $\partial_n \sigma_0|_{\partial\Omega}= 0.$   }
\label{fig:sigma0}
\end{figure}

 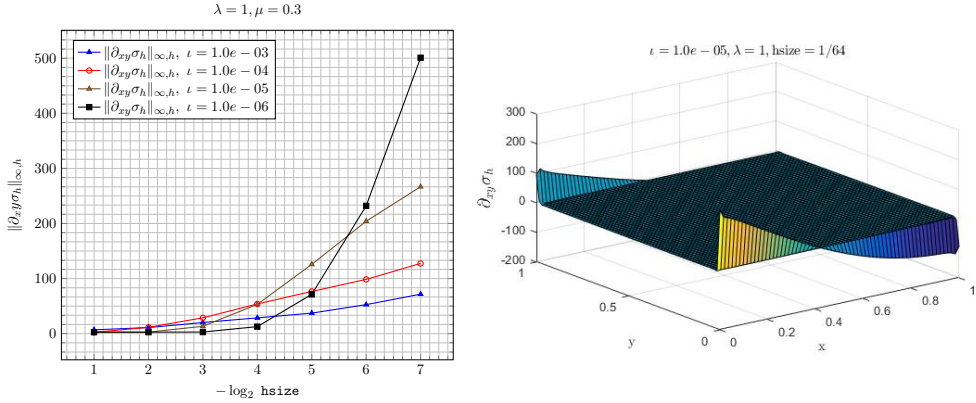
\begin{figure}[htbp]
\begin{minipage}{0.4\textwidth}
\centering
\begin{tikzpicture}[scale = 0.5]
\begin{axis}[
    width=12cm,
  xlabel={$-\log_2~\texttt{hsize}$},
    ylabel={$\|\partial_xy\sigma_h\|_{\infty,h}$},
    legend pos=north west,
    grid=both,
    minor tick num=5,
    title={$\lambda=1,\mu=0.3$ },
    legend cell align={left}
]

\addplot+[mark=triangle*, thick] coordinates {
(1, 6.9026)
(2, 11.0463)
(3, 20.2497)
(4, 28.5922)
(5, 37.3297)
(6, 52.6413)
(7, 71.6812)
};
\addlegendentry{$\|\partial_{xy}\sigma_h\|_{\infty,h},~\iota=1.0e-03$}

\addplot+[mark=o, thick] coordinates {
(1, 2.3661)
(2, 11.5215)
(3, 28.7314)
(4, 53.9067)
(5, 76.6352)
(6, 98.5029)
(7,127.4402)
};
\addlegendentry{$\|\partial_{xy}\sigma_h\|_{\infty,h},~\iota=1.0e-04$ }

\addplot+[mark=triangle*, thick] coordinates {
(1, 2.4791)
(2, 2.6474)
(3, 13.3579)
(4, 53.2579)
(5, 125.712)
(6, 204.0388)
(7, 266.7546)
};
\addlegendentry{$\|\partial_{xy}\sigma_h\|_{\infty,h},~\iota=1.0e-05$ }

\addplot+[mark=square*, thick] coordinates {
(1, 2.6882)
(2, 2.6983)
(3, 2.8897)
(4, 12.7649)
(5, 71.3525)
(6, 231.9887)
(7, 500.8997)
};
\addlegendentry{$\|\partial_{xy}\sigma_h\|_{\infty,h},~\iota=1.0e-06$}

\end{axis}
\end{tikzpicture}
\end{minipage}
\begin{minipage}{0.4\textwidth}
\centering
 \includegraphics[width=1.2\linewidth]{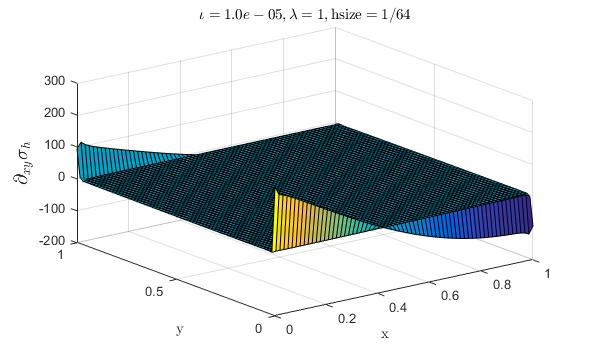}
\end{minipage}

 \caption{The error of  $\|\partial_{xy}\tilde \sigma_{\iota,h}\|_{\infty,h}$ (left) and a profile of $\partial_{xy} \tilde \sigma_{\iota,h}$ (right) for $\partial_n\tilde \sigma_0|_{\partial\Omega}\neq 0.$  }
\label{fig:tildesigma0}
\end{figure}

\section{Concluding Remarks}
\label{sec:discussion}

In this paper, we study the finite element discretization of the linear stress gradient theory. The stability usually relies on the geometric structure. Nevertheless, the following questions still remains open. 
\begin{enumerate}
\item The error estimates at the current stage are not perfect enough. Notably, the asymptotic behavior of the solution as $\iota \to 0$ is absent due to the lack of higher-order regularity results. For higher-order theory, a boudary layer effects would be investigated. It would be interesting to investigate whether the finite element pairs constructed in this paper exhibit the desired asymptotic behavior. 

\item The requirement of the polynomial degree $ k \ge 7$ is far from practice concerns. In fact, all the finite element pair considered in this paper can be constructed with smaller polynomial degrees. The reduction of the polynomial degree requirement will be a part of the future work.
{
\item In this paper we focus on natural boundary condition to simplify the argument while addressing the geometric singularities. For boundary conditions $\sigma n = 0$ on $\Gamma_f \neq \emptyset$, the dimension of discrete space $\Sigma_{h,\Gamma_f}^r$ will be affected by the corner points on $\Gamma_f$, and the stability result will be affected by the singular vertices on $\Gamma_f$. The results can be analyzed similarly as \cite{FalkNeilan13}, and the detailed results are beyond the scope of this paper.
}
\end{enumerate}

\section*{Acknowledgement}
The authors sincerely thank Prof. Dietmar Gallistl for his guidance in discussions and constructive feedback on the manuscript.

\newpage 
\appendix 

\section{Technical Details of \Cref{sec:Sigma0}}
\label{sec:appendix}
We first show the proof of \Cref{lem:necessary-type-II}.
\begin{proof}[Proof of \Cref{lem:necessary-type-II}]
Since $\sigma$ is continuous, and $\{ t_i t_i^T \}_{i=1,2,3}$ forms a basis for $\mathbb S^2$, we express
$\sigma = \alpha t_1 t_1^T + \beta t_2 t_2^T + \gamma t_3 t_3^T,
$ 
where $\alpha$, $\beta$, and $\gamma$ are continuous functions. On each element $T_i$, we express this as
\[
\sigma|_{T_i} = \alpha_i t_1 t_1^T + \beta_i t_2 t_2^T + \gamma_i t_3 t_3^T,
\]
where $\alpha_i = \alpha|_{T_i}(z)$, and $\beta_i$, $\gamma_i$ are similarly defined.

Computing the divergence projected onto $n_{i+1}$, we obtain:
\begin{align*}
\div \sigma|_{T_1}(z) \cdot n_2 &= (n_2 \cdot t_1) \partial_{t_1} \alpha_1 + (n_2 \cdot t_3) \partial_{t_3} \gamma_1 = -\sin \theta_2 \partial_{t_1} \alpha_1 + \sin \theta_3 \partial_{t_3} \gamma_1, \\
\div \sigma|_{T_2}(z) \cdot n_3 &= (n_3 \cdot t_2) \partial_{t_2} \beta_2 + (n_3 \cdot t_1) \partial_{t_1} \alpha_2 = -\sin \theta_3 \partial_{t_2} \beta_2 - \sin \theta_1 \partial_{t_1} \alpha_2, \\
\div \sigma|_{T_3}(z) \cdot n_4 &= (n_4 \cdot t_3) \partial_{t_3} \gamma_3 + (n_4 \cdot t_2) \partial_{t_2} \beta_3 = -\sin \theta_1 \partial_{t_3} \gamma_3 - \sin \theta_2 \partial_{t_2} \beta_3, \\
\div \sigma|_{T_4}(z) \cdot n_5 &= (n_5 \cdot t_1) \partial_{t_1} \alpha_4 + (n_5 \cdot t_3) \partial_{t_3} \gamma_4 = \sin \theta_2 \partial_{t_1} \alpha_4 - \sin \theta_3 \partial_{t_3} \gamma_4, \\
\div \sigma|_{T_5}(z) \cdot n_6 &= (n_6 \cdot t_2) \partial_{t_2} \beta_5 + (n_6 \cdot t_1) \partial_{t_1} \alpha_5 = \sin \theta_3 \partial_{t_2} \beta_5 + \sin \theta_1 \partial_{t_1} \alpha_5, \\
\div \sigma|_{T_6}(z) \cdot n_1 &= (n_1 \cdot t_3) \partial_{t_3} \gamma_6 + (n_1 \cdot t_2) \partial_{t_2} \beta_6 = \sin \theta_1 \partial_{t_3} \gamma_6 + \sin \theta_2 \partial_{t_2} \beta_6.
\end{align*}
Since $\alpha$, $\beta$, and $\gamma$ are continuous, we have $\alpha_i = \alpha_{i+1}$, $\beta_i = \beta_{i+1}$, and $\gamma_i = \gamma_{i+1}$ across adjacent elements. Additionally, $\sin \theta_i = \sin \theta_{i+3}$ due to the geometric structure around $z \in \mathcal{V}_{II}$. Summing the terms in \eqref{eq:lag-cond-typeII}, the contributions of $\partial_{t_j} \alpha_i$, $\partial_{t_j} \beta_i$, and $\partial_{t_j} \gamma_i$ cancel due to the alternating signs and periodicity of $\sin \theta_i$, yielding
\[
\sum_{i=1}^6 (-1)^i \sin \theta_i \, q_i(z) \cdot n_{i+1} = 0,
\]
as required.
\end{proof}

In the rest of this section, we consider the sufficient condition of the degenerate type II singular vertices. This is based on the following identity on imaging meshes:
\begin{equation}
    \sin \theta_{i} n_{i+1} -  \sin \theta_{i+1} n_{i+2} = \sin(\theta_{i+1} + \theta_{i}) n_{i+3}.
\end{equation}

\begin{figure}[htbp]
\begin{tikzpicture}[scale=1.5]
\draw[dashed] (0,0) -- (1,0);
\draw[thick] (1/2,1.732/2) -- (1,0);
\draw[thick] (1,0) -- (1/2,-1.732/2);
\draw[thick] (1,0) -- (3/2,1.732/2);
\draw[thick] (2,0) -- (1,0);
\draw[thick] (1,0) -- (3/2,-1.732/2);
\node [below=] at (3/2,-1.732/2+0.7) {$T_2$};
\node [below=] at (3/2,+0.5) {$T_3$};
\node [below=] at (1,0.8) {$T_4$};
\node [below=] at (0.2,0.2) {$T_{5,6}$};
\node [below=] at (1.,-0.5) {$T_{1}$};
\end{tikzpicture}
\caption{An illustruation for type II-1 singularity.}
\label{fig:typeII-1}

\end{figure}
\begin{lemma}[Sufficient Conditions for Type II-1 Singularity]
\label{lem:typeII-1-sufficient}
Suppose $z \in \mathcal{V}_{II}$ is a Type II-1 singularity (see \Cref{fig:typeII-1}), and $q \in Q_{k-1}^0$ satisfies \eqref{eq:lag-cond-typeII}. Then, there exists $\sigma \in \Sigma_k^0$ such that:
\begin{enumerate}
\item $\sigma$ is supported in $\st(z)$.
\item $\div \sigma|_T(z) = q|_T(z)$ for all $T \in \st(z)$.
\item $\div \sigma|_T(y) = 0$ for all $y \neq z$, $T \in \st(z)$.
\end{enumerate}
Moreover, the bound holds:
\[
\|\sigma\|_{H^1} \lesssim \|q\|_{L^2(\st(z))}.
\]
\end{lemma}

\begin{proof}
For a Type II-1 singularity at $z \in \mathcal{V}_{II}$, the star $\st(z)$ consists of five elements $T_1, T_2, T_3, T_4, T_{5,6}$, where $T_{5,6}$ is a single element sharing edges with $T_4$ and $T_1$, and $n_5 = -n_2$, $\sin \theta_4 = \sin \theta_{5,6}$ (see \Cref{fig:typeII-1}). The condition \eqref{eq:lag-cond-typeII} simplifies to:
\[
\sum_{i=1}^4 (-1)^i \sin \theta_i q_i(z) \cdot n_{i+1} + (-1)^5 \sin \theta_{5,6} q_{5,6}(z) \cdot n_2 = 0.
\]
Since $n_5 = -n_2$ and $\sin \theta_4 = \sin \theta_{5,6}$, this becomes:
\begin{equation}
\label{eq:typeII-1-cond}
\sum_{i=1}^3 (-1)^i \sin \theta_i q_i(z) \cdot n_{i+1} + \sin \theta_4 (q_4(z) - q_{5,6}(z)) \cdot n_2 = 0.
\end{equation}

Recall the function $\xi_r(u)$ from \eqref{eq:xi-r} in \Cref{lem:nonsingular-sufficient}:
\[
\xi_r(u) = \psi_r \psi_z^2 \frac{u \cdot n_{r-1}}{n_{r-1} \cdot t_r} \frac{h_r}{\sin \theta_r} t_r t_r^T + \psi_r \psi_z^2 \frac{u \cdot n_r}{n_r \cdot t_{r-1}} \frac{h_{r-1}}{\sin \theta_r} (t_r t_{r-1}^T + t_{r-1} t_r^T).
\]
For Type II-1 singularities, as in \Cref{lem:typeI-sufficient}, the mesh geometry implies $\sin(\theta_r + \theta_{r+1}) = 0$, so \eqref{eq:div-xi-r} holds: $\div \xi_r(u)|_{T_r}(z) = u$ and $\div \xi_r(u)|_{T_{r+1}}(z) = u$. For $r=4$, since $T_{5,6}$ is adjacent to $T_4$, we have:
\[
\div \xi_4(u)|_{T_4}(z) = u, \quad \div \xi_4(u)|_{T_{5,6}}(z) = u.
\]
Define $q' = q - \div \xi_4(q_{5,6}(z))$. Then:
\begin{itemize}
\item In $T_{5,6}$: $\div \xi_4(q_{5,6}(z))|_{T_{5,6}}(z) = q_{5,6}(z)$, so $q'_{T_{5,6}}(z) = q_{5,6}(z) - q_{5,6}(z) = 0$.
\item In $T_4$: $\div \xi_4(q_{5,6}(z))|_{T_4}(z) = q_{5,6}(z)$, so $q'_{T_4}(z) = q_4(z) - q_{5,6}(z)$.
\item In $T_i$, $i=1,2,3$: $\div \xi_4(q_{5,6}(z))|_{T_i}(z) = 0$, so $q'_{T_i}(z) = q_i(z)$.
\end{itemize}
Thus, $q'$ satisfies \eqref{eq:lag-cond-typeII} with $q'_{T_{5,6}}(z) = 0$.

Now, we can apply a construction similar to \Cref{lem:typeI-sufficient} for the four elements $T_1, T_2, T_3, T_4$, which completes the proof.
\end{proof}

Using similar techniques, we can prove the results of Type II-2 singularity.
\begin{lemma}
    Suppose that $z \in \mathcal V_{II}$ is of Type II-2 (see also \Cref{fig:typeII-2}), and $q$ satisfies \eqref{eq:lag-cond-typeII}. Then there exists $\sigma \in \Sigma_k^0$ satisfying the conditions in \Cref{lem:suff-typeII}.
\end{lemma}
\begin{proof}

For type II-2, \eqref{eq:lag-cond-typeII} reads $$\sin \theta_1 (q_{5,6}(z) - q_1(z)) \cdot n_2 + \sin \theta_2 (q_2(z) - q_{3,4}(z)) \cdot n_3 = 0.$$

We can consider $q' = q - \div \xi_{2}(q_{3,4}(z)) - \div \xi_{6}(q_{5,6}(z))$. Clearly, $q'$ satisfies \eqref{eq:lag-cond-typeII}, but $q'|_{T_{3,4}} = q'|_{T_{5,6}} = 0$. The remaining argument is similar to \Cref{lem:suff-typeII} by using $\sigma_1$, and $\tilde \sigma_1$, $\tilde \sigma_2$. 

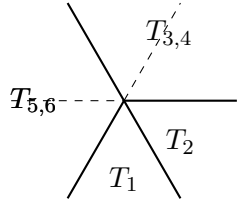
\begin{figure}
				\begin{tikzpicture}[scale=1.5]
\draw[dashed] (0,0) -- (1,0);
\draw[thick] (1/2,1.732/2) -- (1,0);
\draw[thick] (1,0) -- (1/2,-1.732/2);
\draw[dashed] (1,0) -- (3/2,1.732/2);
\draw[thick] (2,0) -- (1,0);
\draw[thick] (1,0) -- (3/2,-1.732/2);
\node [below=] at (3/2,-1.732/2+0.7) {$T_2$};
\node [below=] at (3/2-0.1,+0.8) {$T_{3,4}$};

\node [below=] at (0.2,0.2) {$T_{5,6}$};
\node [below=] at (0.2,0.2) {$T_{5,6}$};
\node [below=] at (1.,-0.5) {$T_{1}$};
		\end{tikzpicture}
        \caption{Illustration for type II-2 singularity.}
        \label{fig:typeII-2}
        \end{figure}

\end{proof}

\begin{figure}[htbp]
\begin{tikzpicture}[scale=1.5]
\draw[dashed] (0,0) -- (1,0);
\draw[thick] (1/2,1.732/2) -- (1,0);
\draw[thick] (1,0) -- (1/2,-1.732/2);
\draw[dashed] (1,0) -- (3/2,1.732/2);
\draw[thick] (2,0) -- (1,0);
\draw[dashed] (1,0) -- (3/2,-1.732/2);
\node [below=] at (3/2-0.1,-1.732/2+0.5) {$T_{1,2}$};
\node [below=] at (3/2-0.1,+0.8) {$T_{3,4}$};

\node [below=] at (0.2,0.2) {$T_{5,6}$};
		\end{tikzpicture}
\caption{Illustration for type II-3 singularity.}
\label{fig:typeII-3}
\end{figure}
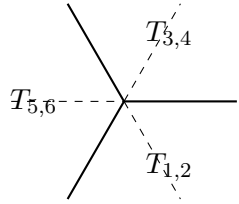

Finally, we consider the Type II-3 singularities.
\begin{lemma}
    Suppose that $z \in \mathcal V_{II}$ is of Type II-3 (see also \Cref{fig:typeII-3}), and $q$ satisfies \eqref{eq:lag-cond-typeII}. Then there exists $\sigma \in \Sigma_k^0$ satisfying the conditions in \Cref{lem:suff-typeII}.
\end{lemma}

\begin{proof}

For type II-3, \eqref{eq:lag-cond-typeII} reads $$\sin \theta_{1,2} q_{1,2}(z) \cdot n_4 + \sin \theta_{3,4} q_{3,4}(z) \cdot  n_{6} + \sin \theta_{5,6}  q_{5,6}(z) \cdot n_{2} = 0.$$ 

Again, we adopt a two-step argument. 
In the first step, we modify the tangential component. Consider $\eta_s = \psi_s\psi_z^2 t_s t_s^T$.
For $\eta_2$, direct calculation yields that 
$$ \div \eta_2|_{T_{1,2}} \cdot n_4 = \frac{\sin \theta_{1,2} \sin \theta_{3,4}}{h_3} = \frac{\sin \theta_{3,4}}{|\overline{zy_2}|}, $$
and 
$$ \div \eta_2|_{T_{3,4}} \cdot n_6 = -\frac{\sin \theta_{1,2} \sin \theta_{3,4}}{h_2} = -\frac{\sin \theta_{1,2}}{|\overline{zy_2}|}, $$

Therefore, by setting $\tau_2 = \frac{|\overline{zy_2}|}{\sin \theta_{1,2} \sin \theta_{3,4}} \eta_2$, we have 

$$\div \tau_s|_{T_{1,2}} \cdot n_{4} \sin \theta_{1,2} = - \div \tau_s|_{T_{3,4}} \cdot n_{6} \sin \theta_{3,4} = 1.$$

    Similarly, we can construct $\tau_4$, $\tau_6$, and an analog argument can ensure that 
 $$(\div \sigma|_{T_{1,2}}(z) - q_{1,2}(z)) \cdot n_{4} = (\div \sigma|_{T_{3,4}}(z) - q_{3,4}(z)) \cdot n_{6} = (\div \sigma|_{T_{5,6}}(z) - q_{5,6}(z)) \cdot n_{2} = 0.$$

In the second step, we modify the tangential components. Now we suppose $\div \sigma|_{T_{1,2}}(z) - q_{1,2}(z) = \alpha_{1,2} t_4.$ Set $\tilde \sigma_{1,2} = \psi_2 \psi_z^2 t_{4} t_{4}^T.$
It then holds that $\div \tilde \sigma_{1,2}|_{T_{3,4}} = 0$, while 
$$\div \tilde \sigma_{1,2}|_{T_{1,2}} = \frac{\sin \theta_{5,6}}{h_{1,2}} t_{4}.$$ Similarly, we can construct $\tilde \sigma_{3,4}$ and $\tilde \sigma_{5,6}$ to finish the proof.

\end{proof}

\bibliographystyle{plain}
\bibliography{ref}

\end{document}